\documentclass[10pt]{amsart}

\usepackage{amsmath, amsthm}
\usepackage{eucal}

\usepackage{amssymb}
\usepackage{amscd}
\usepackage{palatino}

\usepackage{latexsym}
\usepackage{epsfig}
\usepackage{graphicx}
\usepackage{amsfonts}
\usepackage{psfrag}

\usepackage{url}
\usepackage{cite}

\usepackage[margin=1.5in]{geometry}

\input xy
\usepackage[PostScript=dvips,noPS]{diagrams}
\xyoption{all}
\UseComputerModernTips

\numberwithin{equation}{section}
\numberwithin{figure}{section}

\newtheorem{theorem}{Theorem}[section]
\newtheorem{lemma}[theorem]{Lemma}
\newtheorem{proposition}[theorem]{Proposition}
\newtheorem{corollary}[theorem]{Corollary}

\newtheorem{remark}[theorem]{Remark}

\theoremstyle{definition}
\newtheorem{definition}[theorem]{Definition}

\newcommand{\C}{{\mathbb{C}}}

\newcommand{\Z}{{\mathbb{Z}}}
\newcommand{\Q}{{\mathbb{Q}}}
\newcommand{\QQ}{{\mathbb{Q}}}
\newcommand{\R}{{\mathbb{R}}}

\newcommand{\PP}{{\mathbb{P}}}

\newcommand{\sig}{\rm sig\mbox{ }}

\newcommand{\into}{\hookrightarrow}

\newcommand{\calL}{\mathcal{L}}

\newcommand{\CC}{{\mathbb C}}

\newcommand{\ZZ}{{\mathbb Z}}
\newcommand{\poly}{{\rm poly}}

\newcommand{\xxi}{\xi}
\newcommand{\zzeta}{\zeta}

\DeclareMathOperator{\Ker}{Ker}

\DeclareMathOperator{\pt}{pt}
\DeclareMathOperator{\Thom}{Thom}
\DeclareMathOperator{\Todd}{Todd}

\DeclareMathOperator{\im}{Im}

\newarrow{Equalto}{=}{=}{=}{=}{=}

\newcommand{\St}{{\rm St}}

\DeclareMathOperator{\ch}{ch}
\newcommand{\chG}{\mathrm{ch}_G}
\newcommand{\chT}{\mathrm{ch}_T}

\newcommand{\hp}{{\pi}}

\begin{document}

\title{The product structure of the equivariant $K$-theory of the based loop group of $SU(2)$}

\author{Megumi Harada}
\address{Department of Mathematics and
Statistics\\ McMaster University\\ 
Hamilton, Ontario L8S4K1\\ Canada}
\email{Megumi.Harada@math.mcmaster.ca}
\urladdr{\url{http://www.math.mcmaster.ca/Megumi.Harada/}}
\thanks{All authors are partially supported by NSERC Discovery Grants. 
The first author is additionally partially supported by
an NSERC University Faculty Award, and an Ontario Ministry of Research
and Innovation Early Researcher Award.}

\author{Lisa C. Jeffrey}
\address{Department of Mathematics \\
University of Toronto \\ Toronto, Ontario \\ Canada}
\email{jeffrey@math.toronto.edu}
\urladdr{\url{http://www.math.toronto.edu/~jeffrey}} 

\author{Paul Selick} 
\address{Department of Mathematics\\ 
University of Toronto\\ Toronto, Ontario \\ Canada} 
\email{selick@math.toronto.edu} 
\urladdr{\url{http://www.math.toronto.edu/~selick}}

\keywords{equivariant $K$-theory, Lie group, loop group, based loop spaces} 
\subjclass[2000]{Primary: 55N15; Secondary: 22E67}

\date{\today}

%%%%%%%%%%%%%%%%%%%%%
%  Abstract
%%%%%%%%%%%%%%%%%%%%%

\begin{abstract}
Let $G=SU(2)$ and let $\Omega G$ denote the space of continuous based
loops in $G$, equipped with the pointwise conjugation action of
$G$. It is a classical fact in topology that the ordinary cohomology
$H^*(\Omega G)$ is a divided polynomial algebra $\Gamma[x]$. 
The algebra $\Gamma[x]$ can be described as an inverse limit as $k \to
\infty$ of the symmetric subalgebra in $\Lambda(x_1, \ldots,x_k)$
where $\Lambda(x_1, \ldots, x_k)$ is the usual exterior algebra in the
variables $x_1, \ldots, x_k$. We compute the $R(G)$-algebra structure
of the $G$-equivariant $K$-theory $K^*_G(\Omega G)$ which naturally
generalizes the classical computation of $H^*(\Omega G)$ as
$\Gamma[x]$. Specifically, we prove that $K^*_G(\Omega G)$ is an
inverse limit of the symmetric ($S_{2r}$-invariant) subalgebra
$(K^*_G((\PP^1)^{2r}))^{S_{2r}}$ of $K^*_G((\PP^1)^{2r})$, where the
symmetric group $S_{2r}$ acts in the natural way on the factors of the
product $(\PP^1)^{2r}$ and $G$ acts diagonally via the standard action
on each factor. 
\end{abstract}

\maketitle

\setcounter{tocdepth}{1}
\tableofcontents

\section{Introduction}\label{sec:intro}

Let $G$ be a compact connected Lie group and consider the conjugation
action of $G$ on itself. Let $\Omega G$ denote the space of continuous based
loops in $G$, equipped with the pointwise conjugation action of
$G$. 
The ordinary and $G$-equivariant cohomology
rings $H^*(G)$, $H^*(\Omega G)$, $H_G^*(G)$, and $H_G^*(\Omega G)$
were computed decades ago (with contributions from many people), and
these results are by now considered classical; the same is true of the
computations of the ordinary $K$-theory rings $K^*(G)$ and~$K^*(\Omega
G)$. However, somewhat surprisingly, the $G$-equivariant $K$-theory
computation was not addressed in the literature until quite
recently. (For instance, it was only in 2000 that Brylinski and Zhang
computed $K^*_G(G)$ \cite{BryZha00}.) 

The main contribution of this manuscript is a concrete computation of
the $K^*_G$-algebra structure of $K^*_G(\Omega G)$ for
the specific case $G =SU(2)$. (We addressed the additive,
i.e. $K^*_G$-module, structure of $K^*_G(\Omega G)$ in our 
companion paper \cite{HJS12}.) 
More specifically, we describe this product structure in a concrete
manner which is a straightforward and pleasant generalization of the
classical fact 
that the ordinary 
cohomology ring $H^*(\Omega G)$ is a divided polynomial algebra
$\Gamma[x]$. To make this more precise, let us briefly recall several
descriptions (some quite familiar, and some perhaps not so) of
$\Gamma[x]$. 
Firstly, $\Gamma[x]$ can be described as the group additively
generated by elements labelled $\gamma_k(x)$ (with degree $k|x|=2k$)
satisfying the multiplicative relations 
\[
\gamma_i(x) \gamma_j(x) = \binom{i+j}{i} \gamma_{i+j}(x).
\]
Alternatively, $\Gamma[x]$ is the subring of $\Q[x]$ generated by the
elements $\{\frac{x^k}{k!}\}$ for varying $k$. A third description is
as the Hopf algebra dual of $\Z[x]$. Finally, yet one more description
is obtained by observing that 
\[
\Gamma[x]/\{\gamma_j(x): j > k\} \cong
\bigl(\Lambda(x_1, \ldots, x_k)\bigr)^{S_k}\cong
 \{ \textup{ symmetric polynomials
  in } \Lambda(x_1, \ldots, x_k) \}
\]
where $\Lambda(x_1,\ldots,x_k)$ is the exterior algebra in the
variables. The ring $\Gamma[x]$ can then be identified with the graded inverse limit
as $k\to \infty$.
Our computation of $K^*_G(\Omega G)$ is a suitable generalization
of this last description of $\Gamma[x]$. Thus, this manuscript proves what in some sense is a ``classical''
theorem in topology: its statement could have been made and 
understood by the topologists working decades ago, and the computation
itself fits naturally with the classical results. Nevertheless, since
the literature appears to have been silent on this issue, we have
taken this opportunity to provide the details. 

The following is our main result. 

\begin{theorem}\label{intro:main}
Let $G=SU(2)$ and let $\Omega G$ be the space of (continuous) based
loops in $G$, equipped with the natural $G$-action by pointwise
conjugation. Then 
$$K_G(\Omega G)
=\varprojlim_r\, 
\Bigl(K_G\bigl((\PP^1)^{2r}\bigr)\Bigr)^{S_{2r}}
=\varprojlim_r\, 
\{\mbox{symmetric polynomials in~$K_G\bigl((\PP^1)^{2r}\bigr)$}\}
$$
where 
$K_G\bigl((\PP^1)^{2r}\bigr)\cong 
R(G)[L_1, \dots, L_{2r}]/I,$
and $I$ is the ideal generated by $\{L_j^2-vL_j+1\}_{j=1}^n$.
Here $R(G)$ is the representation ring of $G$,
$v$ is the standard representation of $G = SU(2)$ on~${\CC}^2 $ and $L_j$ is
the pullback of either the canonical line bundle over the
$j$th factor of~$(\PP^1)^{2r}$ or its inverse, depending on~$j$ (see
Definition \ref{definition:L_j and barL_j}). 
The system maps in this inverse system are given by
$$i^* (s_j) = \begin{cases}
s'_0 &\mbox{if $j=0$};\cr
s'_1 +v s'_{0} &\mbox{if $j=1$};\cr
 s'_j +v s'_{j-1} + s'_{j-2}
& \mbox{if $1 < j \le 2r-2$};\cr
  v s'_{2r-2} + s'_{2r-3}
& \mbox{if $ j = 2r-1$};\cr
       s'_{2r-2}
& \mbox{if $    j =2r$},\cr
\end{cases}
$$
where $s_j$ and $s'_j$ are the $j$th elementary symmetric polynomials in
$\{L_1,\ldots, L_{2r}\}$ and  respectively 
$\{L'_1,\ldots, L'_{2r-1}\}$.
(See equation~\ref{eq:iota star of s_j}.)
\end{theorem}

We now briefly sketch the outline of our proof. From the module
calculations in \cite{HJS12} we know that $K^*_G(\Omega G)$ is the
inverse limit of $K^*_G(F_{2r})$ as $r \to \infty$ for a certain
$G$-invariant filtration $F_0 \subseteq F_2 \subseteq \cdots \subseteq
F_{2r} \cdots$ of $\Omega G$. Thus, in the present manuscript, we
focus on the computation of the $K^*_G$-algebra structure of
$K^*_G(F_{2r})$. To accomplish this, we consider the $G=SU(2)$-space
$(\PP^1)^{2r}$ for each $r \geq 0$, where the $G$ acts diagonally on
each factor in the usual way, induced from the standard representation
of $SU(2)$ on $\C^2$. 
We define maps 
$\Phi_{2r}: (\PP^1)^{2r} \to F_{2r}$ for each $r\geq 0$ in
Section~\ref{sec:about PP^1^n} and then prove in 
Proposition~\ref{proposition:injective in KG} that the induced map
\[
\Phi_{2r}^*: K^*_G(F_{2r}) \to K^*_G((\PP^1)^{2r})
\]
is injective for all $r>0$. We also give an explicit and convenient presentation of
the right hand side via generators and relations in
Theorem~\ref{theorem:P1summary}. 
The main (and longest) technical argument
in this manuscript is a computation of the image of $\Phi_{2r}^*$ in
terms of the natural generators of $K^*_G((\PP^1)^{2r})$ in a
manner analogous to the description of $H^*(\Omega G)$ as a divided
polynomial algebra (Theorem~\ref{KF2r}). Taking the inverse limit of
this description yields $K^*_G(\Omega G)$. At various points in 
our argument, we find it necessary or useful to first prove the corresponding
statements in equivariant cohomology, and then use these to deduce the analogous
results in equivariant $K$-theory.

\bigskip

\noindent \textbf{Notation and standard facts.} 

\begin{itemize}
\item $T$ denotes the maximal torus of $G=SU(2)$ given by 
$$\left\{\begin{pmatrix}z&0\cr0&z^{-1}\end{pmatrix}
\mathop{\Big\vert} z\in S^1\right\}. $$
\item $W\cong S_2$ is the Weyl group of~$G$ and $w\in W$ is the nontrivial element.
\item $R(T) \cong \Z[b,b^{-1}]$ is the representation ring of~$T$, where $b$ is the
weight~$1$ (identity map) one-dimensional representation of~$T$,
and $w(b)=b^{-1}$.
\item $v \in R(G)$ is the standard (two-dimensional) representation of~$G$
on~$\C^2$. The restriction of $v$ to~$T$ is~$b\oplus b^{-1}$.
\item $R(G) \cong \Z[b,b^{-1}]^W = \Z[v]$ where $v = b+b^{-1}$.
\item $K^*_T(\pt) \cong R(T) \cong \Z[b,b^{-1}]$ and $K^*_G(\pt) \cong R(G) \cong \Z[v]$.
\item We let $H^\hp(X)$ denote the direct product $\prod_{i=0}^\infty H^i(X) $ and
$\tilde{H}^\hp(X)$ the direct product $\prod_{i=1}^\infty H^i(X)$. 
\item We will systematically use bars to denote elements in equivariant
cohomology analogous to the corresponding letter for equivariant
$K$-theory.

\end{itemize}

\bigskip
\noindent \textbf{Acknowledgements.} We thank Greg Landweber and 
Eckhard Meinrenken for
many useful discussions.

\section{Background}\label{sec:background}

In this section, we assemble for the convenience of the reader brief
accounts of various definitions and constructions used in later
sections.

\subsection{Polynomial loops and a filtration on $\Omega G$}
\label{subsec:polyloops}

Following \cite{PS86}, we define the space of \textbf{polynomial based loops
$\Omega_{\poly}U(n)$} as 
the set of maps $S^1 \to
U(n)$ based at the identity which can be expressed as polynomials in $z$ (here $z$ is 
the parameter on the circle $S^1$).
More precisely, we define 
\begin{equation} \label{polydef}
\Omega_{\poly,r}U(n) :=\left\{f: S^1 \to U(n)
\ \Bigg|\  f(1) = \mathbf{1}_{n \times n},\ f=\!\sum_{j=-r(n-1)}^ra_jz^j,
\space a_j \in M(n\times n,\C)
\right\}
\end{equation}
for a positive integer $r$, where $\mathbf{1}_{n \times n}$ denotes
the identity matrix and $M(n \times n,\C)$ is the space of $n\times n$
complex matrices. 
Note that the $a_j$ are constant matrices, and
$f(z)$ is required to be unitary (in particular invertible) for all
$z \in S^1 $.
An element in $\Omega_{\poly,r}U(n)$ may also be viewed as an
element of $\Omega_{\poly, r'}U(n)$ for any $r' > r$; via these
natural inclusions we now define
\[
\Omega_{\poly}U(n) := \bigcup_{r=0}^\infty \Omega_{\poly,r} U(n). 
\]
We define $\Omega_{\poly}H \subseteq \Omega_{\poly}U(n)$ for any subgroup $H \subseteq U(n)$ by
requiring the images $f(z)$ of an element $f \in \Omega_{\poly}U(n)$
to lie in $H$ for all $z \in S^1$.

By definition, the space $\Omega_{\poly} H$ comes equipped with a
filtration given by the successive subspaces $\Omega_{\poly,r}H$. 
For the case $H = SU(2)$, we write $F_{2r} $ for $\Omega_{\poly,r} SU(2)$; the motivation for this
notation comes from the fact that there exists a subspace of the
Grassmannian, denoted $F_{2r}$ in \cite{HJS12}, which is $SU(2)$-homotopy
equivalent to $\Omega_{\poly,r}SU(2)$
\cite[Thm 2.9(3), Equation (3.1)]{HJS12}. 

\begin{remark} The reason we index our filtration  by only 
the \emph{even} integers 
is related to the need to create a $SU(2)$-equivariant 
filtration, as discussed in detail in \cite{HJS12}. It is related to 
the fact that $\Omega_{\poly,r} SU(2)$ contains 
Laurent polynomials
whose degrees run from $-r$ to $r$. 
An analogous non-equivariant filtration indexed by 
\emph{all} integers is discussed by Pressley
and Segal in \cite{PS86}. The distinction will not be important for
the present paper, so we do not discuss this further. 
\end{remark}

Finally, note that matrix multiplication induces a map 
\begin{equation} \label{matmult}
F_{2j} \times F_{2k} \to 
F_{2(j+k)}.
\end{equation}
This map will be crucial to our constructions below.

\begin{remark}
In  \cite[Theorems 5.3-4]{HJS12} it is shown that 
 there is an $SU(2)$-homotopy equivalence 
between $\Omega_{\poly} SU(2)$ and $\Omega SU(2)$.
\end{remark}

\subsection{The equivariant Thom space and the equivariant Thom class
  in $K$-theory} \label{background:Atiyah}

Let $H$ be a compact Lie group, and let $\xi$ be an $n$-dimensional
complex vector bundle over a base space $X$ equipped with 
an action of $H$.
Let $p: E(\xi) \to X$ denote 
the bundle projection and let $\tilde{p}:=p\bigl(\PP(\xi)\bigr):
\PP(\xi) \to X$ denote the associated projective bundle with fiber
$\PP^{n-1}$. 
We let $\gamma_\xi$ denote the canonical line bundle
over~$\PP(\xi)$. We have the following \cite{Atiyah}.

\begin{lemma}\label{lemma:atiyah}
Let $H$ and $\xi$ be as above. Then there exists an $H$-bundle
$\beta_{\xi}$ such that 
$\tilde{p}^*(\xxi)\cong \gamma_\xi \oplus\beta_{\xi}$
as $H$-equivariant bundles.
\end{lemma}

\begin{proof}
The bundle $\gamma_{\xi}$ is naturally an $H$-subbundle of
$\tilde{p}^*(\xi)$, so by using a choice of 
$H$-invariant Riemannian metric, we can define $\beta_{\xi}$ to be the
orthogonal complement of $\gamma_{\xi}$. This $\beta_{\xi}$ has the
required property. 
\end{proof}

We next recall a convenient description of the (equivariant) Thom space of a bundle, following
Atiyah \cite[p.100]{Atiyah}.
Although Atiyah does not explicitly say so, if $\xi$ is an $H$-bundle, then all the maps in 
his description are $H$-equivariant.

\begin{proposition} \label{p:atiyah}
Let $H$ be a compact Lie group and let $\xi$ be a finite-dimensional
complex vector bundle over a base space $X$ equipped with an action of
$H$. Let $\epsilon$ denote the trivial complex line bundle over $X$
equipped with the trivial action on the fibers. Then $\Thom(\xi) \cong_H
\PP(\xi \oplus \epsilon)/\PP(\xi) $.
\end{proposition}

\begin{proof} 
By definition, $\Thom(\xi)$ is $D(\xi)/S(\xi)$ where $D(\xi)$ and
$S(\xi)$ denote the disk and sphere bundles associated to $\xi$
respectively. Note that we can identify the disk bundle as
$$D(\xi) \cong \{ tv\mid v \in S(\xi), t \in [0,1] \} $$
while we can identify 
$$\PP(\xi\oplus \epsilon) 
\cong \{ (u,w) \mid u \in D(\xi), w \in \CC, |u|^2 + |w|^2 = 1\}/\sim $$
where the equivalence relation $\sim$ is given by: $(\lambda u,
\lambda w) \sim (u,w)$ for $\lambda \in S^1 $. With these
identifications, consider 
the map $D(\xi) \to \PP(\xi \oplus \epsilon) $ given by 
$tv \mapsto [tv, \sqrt{1-t^2} ].$ 
This map is $H$-equivariant, and it restricts to an $H$-equivariant map 
$S(\xi) \to \PP(\xi)$ by taking $t = 1$. 
Therefore the induced map $\PP(\xi \oplus \epsilon) \to \Thom(\xi)$
in the pushout diagram
\begin{equation} \label{thomdiagram}
\begin{diagram}
S(\xi)  &\rTo&   \PP(\xi)   \cr
\dTo && \dTo_{j} \cr
D(\xi)  &\rTo&     \PP(\xi\oplus \epsilon)  \cr
\dTo && \dTo \cr
\Thom(\xi) &\rEqualto& \Thom(\xi) \cr
\end{diagram}
\end{equation}
is $H$-equivariant. 
\end{proof}

Under the assumption that $X$ and $BH$ have cells only in 
even degree (which will be satisfied by the spaces we study),
from Proposition~\ref{p:atiyah} it follows that we have short exact sequences
\begin{equation} \label{exseq} 0 \to \tilde{Y}^*\bigl(\Thom(\xi)\bigr) \to
\tilde{Y}^*\bigl(\PP(\xi \oplus \epsilon)\bigr) \rTo^{j^*}
\tilde{Y}^*\bigl(\PP(\xi )\bigr) \to 0 \end{equation}
where $Y^*$ here denotes any of the following generalized
(equivariant) cohomology theories: $K_H^*$; $H_H^*$; $K^*$ and $H^*$.
The following is immediate and we will henceforth always use this
identification.

\begin{lemma}\label{lemma:identify kernel j}
Let $H$ and $\xi$ be as above. Then $\tilde{Y}^*\bigl(\Thom(\xi)\bigr) \cong 
\Ker j^*$ via the sequence~\eqref{exseq}. 
\qed
\end{lemma}

We now specialize the discussion to equivariant $K$-theory. In this
setting the equivariant Thom class $U_\xi \in
\tilde{K}^*_H(\Thom(\xi))$ of $\xi$ can be described as follows
\cite[p.102-103]{Atiyah}. We have 
\begin{equation}\label{def eqvt Thom class} 
U_\xi = \lambda\bigl(\gamma^*_{\xi \oplus \epsilon} \otimes \bar{p}^* (\xi)\bigr) \in
\Ker j^* \cong \tilde{K}^*_H\bigl(\Thom(\xi)\bigr ) 
\end{equation}
where $\bar{p}:\PP(\zeta \oplus \epsilon) \to X$
is the bundle projection, $\lambda$ is the exterior bundle and we identify
$\tilde{K}^*_H(\Thom(\xi)) \cong \Ker j^*$ as in the lemma above. 

Let $\zzeta$ be a rank $n$ complex vector bundle with base
space $X$, equipped with an action of~$H$. 
Consider the following diagram
\begin{equation}\label{maindiagram}
\begin{diagram}
\PP^{n-1}&\rTo&\PP(\zzeta)&\rTo^{\tilde{p}}&X\cr
\dTo&&\dTo_j&&\dEqualto\cr
\PP^n&\rTo^{k}&\PP(\zzeta\oplus\epsilon)&\rTo^{\bar{p}}&X\cr
&&\dTo&&\cr
&&\Thom(\zzeta)&&\cr
\end{diagram}
\end{equation}
where $\tilde{p}: \PP(\zzeta) \to X$ and $\bar{p}: \PP(\zzeta \oplus
\epsilon) \to X$ denote the bundle projections and $$j: \PP(\zzeta)
\into \PP(\zzeta \oplus \epsilon) \qquad 
\mbox{and}\qquad k: \PP^n \into \PP(\zzeta
\oplus \epsilon)$$ denote the natural inclusions.
 We will use this
diagram repeatedly. 

\begin{remark}
In order for the inclusion $k: \PP^1 \into \PP(\zzeta \oplus \epsilon)$ 
to be an $H$-equivariant map, the image $\bar{p}(k(\PP^1)) \in X$ must
be an $H$-fixed point.
In some of our applications of the above diagram, the space~$X$ will not have
an $H$-fixed point. In these cases, the map $k$ is only a map of topological
spaces, since the fiber has no $H$-action. 
\end{remark}

Applying the argument for Lemma~\ref{lemma:atiyah} to $\zzeta \oplus
\epsilon$ and using the fact that $\bar{p}^*(\zzeta \oplus \epsilon)
\cong \bar{p}^*(\zzeta) \oplus \epsilon$, we obtain the following. 

\begin{lemma}\label{lemma:beta}
Let $H$ and $\zzeta$ be as above. Then there exists a complex line
bundle $\beta_{\zzeta \oplus \epsilon}$ over $\PP(\zzeta \oplus
\epsilon)$ equipped with an action of $H$ such that 
\begin{equation}\label{basemainequation}
\bar{p}^*(\zzeta) \oplus\epsilon \cong 
\gamma_{\zzeta\oplus\epsilon}\oplus\beta_{\zzeta\oplus \epsilon}
\end{equation}
as $H$-vector bundles.
\qed
\end{lemma} 

In what follows, we will often use the above results for the 
special case in which $\zzeta$  is a complex line bundle, i.e. $n=1$,
so we take a moment to discuss this case further. 
In this special case, notice that $\tilde{p}$ is the identity map, since $\PP^{n-1}=\PP^0$ is a
point.  This means the composition $\bar{p} \circ j$ can be identified
with the identity map on $X$, in turn implying that 
the map $j^*$ splits.  We conclude that there exists a direct sum decomposition
$$\tilde{K}^*_H\bigl(\PP(\zzeta\oplus\epsilon)\bigr)
\cong \tilde{K}^*_H\bigl(\Thom(\zzeta)\bigr)\oplus\tilde{K}^*_H(X
\cong \PP(\zzeta))$$
where again by Lemma~\ref{lemma:identify kernel j} we identify $\tilde{K}^*_H\bigl(\Thom(\zzeta)\bigr )\cong \Ker j^*$.

Using the description of the equivariant Thom class in~\eqref{def eqvt
  Thom class}, we now give a concrete computation of $U_\zzeta$ for an
$H$-equivariant line bundle $\zzeta$. 

\begin{lemma}\label{lemma:compute Uzzeta}
  For $H$ and $\zzeta$ as above, the $H$-equivariant $K$-theoretic
  Thom class $U_\zzeta$ is given by 
\[
U_\zzeta = 1 - \beta_{\zzeta \oplus \epsilon} \in \Ker j^* \cong
\tilde{K}^*_H(\Thom(\zzeta)).
\]
\end{lemma}

\begin{proof}
  By~\eqref{def eqvt Thom class} we have that 
\[
U_\zzeta = \lambda(\gamma^*_{\zzeta \oplus \epsilon} \otimes
\bar{p}^*(\zzeta)).
\]
Recall that $H$-equivariant line bundles are classified by their
$H$-equivariant first Chern class. The
definition~\eqref{basemainequation} of $\beta_{\zzeta \oplus
  \epsilon}$ implies that 
\[
c_1^H(\bar{p}^*(\zzeta)) = c_1^H(\gamma_{\zzeta \oplus \epsilon}) +
c_1^H(\beta_{\zzeta \oplus \epsilon})
\]
which in turn implies $\beta_{\zzeta \oplus \epsilon} \cong
\gamma^*_{\zzeta \oplus \epsilon} \otimes \bar{p}^*(\zzeta)$. Thus
$U_\zzeta = \lambda(\beta_{\zzeta \oplus \epsilon})$, and since
$\beta_{\zzeta \oplus \epsilon}$ is a line bundle, its exterior power
is simply $1 - \beta_{\zzeta \oplus \epsilon}$, as desired. 
\end{proof}

\begin{remark}
  Note that $1 - \beta_{\zzeta \oplus \epsilon}$ can be checked to lie
  in $\Ker j^*$ by the following direct argument. By definition of the
  tautological line bundle, it is straightforward to see that
  $j^*(\gamma_{\zzeta \oplus \epsilon})=\gamma_\zeta \cong \zzeta$
where the final isomorphism follows because $\bar{p}\circ j = \tilde{p}$
can be identified with the identity map.
The defining equation~\eqref{basemainequation} then implies 
  \begin{equation}
    \begin{split}
      \zzeta \oplus \epsilon & \cong j^* \bar{p}^*(\zzeta) \oplus
      \epsilon \\
 & \cong j^*(\gamma_{\zzeta \oplus \epsilon}) \oplus j^*(\beta_{\zzeta
   \oplus \epsilon}) \\
 & \cong \zeta \oplus j^*(\beta_{\zzeta \oplus \epsilon}) \\
    \end{split}
  \end{equation}
which in turn implies $j^*(1 - \beta_{\zzeta \oplus \epsilon}) = 0 \in
  \tilde{K}^*_H(\PP(\zzeta))$. 
\end{remark}

\subsection{The bundle $\tau$}\label{subsec:tau}

We now briefly recall the definition and some of the properties of a
bundle $\tau$ discussed in \cite{HJS12}. 
The main reason for its appearance in this paper is
Proposition~\ref{prop:filtration quotient as Thom space}.

Let $G = SU(2)$. Let $\gamma$ denote the tautological bundle over $\PP^1$ and let
  $\perp$ denote the orthogonal complement with respect to the
  standard metric in $\C^2$.
We have the following \cite[Definition 3.1]{HJS12}. 

\begin{definition}\label{definition:tau}
We define $\tau$ to be the $G$-equivariant
complex line bundle over~$\PP^1$ with total space
$$E(\tau)=
\bigl\{(u,v)\mid u\in S^3\subset\C^2, v\in (u^\perp)\bigr\}
/\mathord{\sim}$$
where the equivalence relation is given by $(u,v)\sim(\zeta u,\zeta
v)$ for $\zeta\in S^1$ and the bundle projection to $\PP^1$ is given by 
taking the first factor. 
\end{definition}

The bundle $\tau$ can in fact be identified as the tangent bundle to
$\PP^1$ \cite[Proposition 3.2]{HJS12}. 

\begin{proposition}\label{prop:tau}
The bundle $\tau$ of Definition~\ref{definition:tau} is
$G$-equivariantly isomorphic to
the tangent bundle $T\PP^1$ of~$\PP^1$.
\qed
\end{proposition}

The following is immediate. 

\begin{corollary}\label{cor:Cherntau}
The $G$-equivariant first Chern class of $\tau$ is given by
$c_1^G(\tau)=-2c_1^G(\gamma)$.
In particular, $\tau \cong_G \gamma^{-2} \cong_G (\gamma^*)^2. $

\end{corollary}

\begin{proof}
The corresponding statement for $T\PP^1$ is \cite[Theorem~14.10] {Mil74},
where the $a$ in the statement of that theorem is identified in its proof
as~$c_1(\gamma)$. (Here we used the fact that $H_G^2(X) \cong H^2(X)$ since $G=SU(2)$ is
simply connected. See the discussion below Remark \ref{rk1}.)
\end{proof}

\begin{remark}\label{rem:Milnor}
The proof of \cite[Theorem~14.10] {Mil74}, to which we refer in the
above proof, contains the assertion that
$\tau\oplus\epsilon\cong 2\gamma^*$ as topological bundles, where $\gamma^*$
denotes the dual bundle to~$\gamma$.
The corresponding $G$-equivariant statement is as follows. Consider
the standard action of $G=SU(2)$ on $\PP^1$ and $\C^2$
respectively. Equip 
$\PP^1 \times \C$ with the diagonal
$SU(2)$-action. Projection to the first factor $p: \PP^1 \times \C^2
\to \PP^1$ makes this a $G$-equivariant bundle which is topologically
trivial but $G$-equivariantly non-trivial. We denote this rank $2$
$G$-bundle by $v$, and by slight abuse of language, call it the
standard $2$-dimensional representation of $G$. The $G$-equivariant
analogue of the non-equivariant statement $\tau \oplus \epsilon \cong
2 \gamma^*$ above is then 
\begin{equation}
  \label{eq:tau plus epsilon is v times gammadual}
  \tau \oplus \epsilon \cong v \otimes \gamma^*.
\end{equation}
\end{remark}

The bundle $\tau$ is an essential ingredient 
in our description of the filtration quotient 
$\frac{F_{2r}}{F_{2(r-1)}},$ where the $F_{2r}$'s are the subspaces of
$\Omega_{\poly}G$ introduced in Section~\ref{subsec:polyloops}. 
The following is \cite[Proposition 3.4]{HJS12}. 

\begin{proposition}\label{prop:filtration quotient as Thom space}. 
Let \(r \in \Z\) and \(r \geq 0.\) The quotient space $F_{2r}/F_{2r-2}$ is
$G$-equivariantly homeomorphic to  $\Thom(\tau^{2r-1})$. In
particular, $F_2 \cong \Thom(\tau)$. 
\end{proposition}

\begin{remark}\label{remark:F2}
  It also follows  from the proof of
  \cite[Proposition 3.4]{HJS12} that, under the identification $F_2
  \cong \Thom(\tau)$, the `point at infinity' of the Thom space gets
  identified with the subspace~$F_0$ which is a single point consisting of
the constant map at the identity in $F_2 = \Omega_{\poly,1}SU(2)$. 
\end{remark}

\subsection{Equivariant Bott periodicity}\label{subsection:BottP}

Here we set some notation and recall 
a special case of the equivariant Bott periodicity theorem, which will be used in Section~\ref{sec:about PP^1^n}.
A reference for this section is ~\cite{AtiyahSegal:1965}.

Recall that the 
classical (non-equivariant) Bott periodicity theorem relates the $K$-theory of
the Cartesian product $X \times S^2 \cong X \times \PP^1$ with that of~$X$.
In the equivariant
case, the analogous theorem relates the equivariant
$K$-theory of $\PP(\xi \oplus \epsilon)$ with that of $X$, where $\xi$ is
an equivariant bundle over $X$. 
More specifically, let $H$ be a compact Lie group. Let $\xi$ be
a complex vector bundle 
over a base space $X$ equipped with an
action of~$H$.
When no confusion will arise,
we denote by $\xxi$ the class $[\xxi]$ in~$K^*_H(X)$.
In fact, we will consider the special case
in which $\xxi=\zzeta\oplus\epsilon$, where
$\zzeta$ is a complex line bundle equipped with an action of $H$. 
We already saw in Section~\ref{background:Atiyah} that, in this
situation, 
\[
\gamma_{\zzeta \oplus \epsilon}^* \otimes \bar{p}^*(\zzeta) \cong
\beta_{\zzeta \oplus \epsilon}
\]
or equivalently 
\[
\gamma_{\zzeta \oplus \epsilon} \cong \bar{p}^*(\zzeta) \otimes
\beta^*_{\zzeta \oplus \epsilon}
\]
where $\beta_{\zzeta \oplus \epsilon}$ is the (line) bundle found in
Lemma~\ref{lemma:beta}. Recalling that $\beta^*_{\zzeta \oplus
  \epsilon} \cong \beta^{-1}_{\zzeta \oplus \epsilon}$ for line
bundles, the defining equation~\eqref{basemainequation} becomes 
\begin{equation}\label{mainequationmod}
\bar{p}^*(\zzeta) \oplus \epsilon \cong (\bar{p}^*(\zzeta) \otimes
\beta^{-1}_{\zzeta \oplus \epsilon}) \oplus \beta_{\zzeta \oplus
  \epsilon}.
\end{equation}
This implies the equality 
\[
\bar{p}^*(\zzeta)+1 = \bar{p}^*(\zzeta)\beta^{-1}_{\zzeta \oplus
  \epsilon} + \beta_{\zzeta \oplus \epsilon}
\]
in ${K}^*_H(\PP(\zeta\oplus \epsilon))$. Some simple
manipulation then yields the relation 
\[
(\beta_{\zeta \oplus \epsilon} - \bar{p}^*(\zzeta))(\beta_{\zzeta
  \oplus \epsilon} - 1) = 0
\]
in ${K}^*_H(\PP(\zeta \oplus \epsilon))$, or equivalently, 
\[
(\gamma_{\zzeta \oplus \epsilon}-1)(\gamma_{\zzeta \oplus \epsilon} -
\bar{p}^*(\zzeta)) = 0. 
\]
We have just explicitly derived a relation in $K^*_H(\PP(\zzeta \oplus
\epsilon))$, but in fact, the equivariant Bott periodicity theorem
(cf. \cite[Theorem 2.7.1]{AtiyahSegal:1965} or \cite[Theorem 3.2]{Greenlees}),
applied to this case, states that this is in fact the only one. 
More precisely, we have the following. 
\begin{theorem}\label{Bott}
Let $H$ be a compact Lie group and let $\zzeta$ be an $H$-equivariant
complex line bundle over a base space $X$. Let $\gamma_{\zzeta \oplus
  \epsilon}, \beta_{\zzeta \oplus \epsilon}$ be as above. Then 
\[
K^*_H\bigl(\PP(\zzeta\oplus\epsilon)\bigr)\cong
\frac{K^*_H(X)[\gamma_{{\zzeta}\oplus\epsilon}]}
{(\gamma_{{\zzeta}\oplus\epsilon}-1)
(\gamma_{{\zzeta}\oplus\epsilon}-\bar{p}^*(\zzeta))}
\cong
\frac{K^*_H(X)[\beta_{\zzeta \oplus \epsilon}]}{(\beta_{\zzeta \oplus \epsilon}-\bar{p}^*(\zzeta))(\beta_{\zzeta \oplus \epsilon}-1)}
\]
as $R(H)$~algebras.
\end{theorem}

We will also use the following (cf. \cite[Theorem
6.1.4]{AtiyahSegal:1965} or \cite[Theorem 3.1]{Greenlees}), which is a
special case of the equivariant Thom isomorphism theorem. 

\begin{theorem} \label{thomisom} 
Let $H$ be a compact Lie group and let $\zzeta$ be an $H$-equivariant
complex line bundle over a base space $X$. Let $U = U_\zzeta$ denote
the equivariant Thom class associated to $\zzeta$ as in~\eqref{def
  eqvt Thom class}. Then 
the multiplication map $x\mapsto U\cdot \bar{p}^*(x)$ defines an isomorphism
$$K_H(X)\cong U\cdot {\bar{p}}^*\bigl(K_H(X)\bigr)
\cong \Ker j^*\cong\tilde{K}_H\bigl(\Thom(\zzeta)\bigr).$$
\end{theorem}

\section{The equivariant cohomology and equivariant $K$-theory 
of~$(\PP^1)^n$}\label{sec:about PP^1^n}

The main purpose of this section is two-fold. First, we define a
$G$-equivariant continuous map $\Phi_{2r}: (\PP^1)^{2r} \to F_{2r}$.
Second, we give convenient presentations of both $K^*_T((\PP^1)^n)$
(respectively $K^*_G((\PP^1)^n)$) and $H^*_T((\PP^1)^n;\Q)$ (respectively
$H^*_G((\PP^1)^n;\Q)$) via generators and relations. In later sections,
we show that the pullback maps 
\[
\Phi_{2r}^*: H^*_T(F_{2r};\Q) \to H^*_T((\PP^1)^{2r};\Q)
\]
(respectively from $H^*_G(F_{2r};\Q)$ to $H^*_G((\PP^1)^{2r};\Q)$) and 
\[
\Phi_{2r}^*: K^*_T(F_{2r}) \to K^*_T((\PP^1)^{2r})
\]
(respectively from $K^*_G(F_{2r})$ to $K^*_G((\PP^1)^{2r})$) are
injective, and use the concrete presentations of the codomains given
here, in order to give explicit computations of $H^*_G(F_{2r};\Q)$ and,
most importantly, 
$K^*_G(F_{2r})$. This is a key step towards our main goal, which is
the computation of the algebra structure of $K^*_G(\Omega G)$.

We begin with the construction of the maps $\Phi_{2r}$. 
Let $\tau$ be the $G=SU(2)$-equivariant bundle of 
Definition~\ref{definition:tau}. 
We begin with an equivariant identification of the $G$-spaces
$\PP(\tau~\oplus~\epsilon)$ and $\PP^1 \times \PP^1$, where $G$ acts
on $\PP^1 \times \PP^1$ via the diagonal action. (The
non-equivariant version of the statement below is of course standard.
We choose to record the equivariant version since we did not find a
reference in the literature.) We prove the lemma by constructing an
explicit homeomorphism which can easily be seen to be $G$-equivariant.

\begin{lemma}\label{lemma:PtauplusEpsilon-prodP1s}
There exists a $G$-equivariant homeomorphism
$\Theta: \PP(\tau \oplus \epsilon) \to \PP^1 \times \PP^1 $.
Under this homeomorphism, the subspace $\PP(\tau)$ corresponds to the diagonal
in~$\PP^1\times\PP^1$, and the bundle projection
$\PP(\tau\oplus\epsilon) \to \PP^1$
corresponds to the projection of $\PP^1\times \PP^1$ onto the first factor.
\end{lemma}

\begin{proof} 
Recall that the center $\{\pm 1\} \cong \Z/2\Z$ of $SU(2)$ acts trivially on both
$\PP(\tau \oplus \epsilon)$ and $\PP^1 \times \PP^1$, hence 
the $SU(2)$ action factors
through an $SO(3)$-action. It is
well-known that the tangent bundle $\tau$ to $\PP^1$ is
$SO(3)$-equivariantly homeomorphic to the tangent bundle $TS^2$
of the
unit sphere $S^2$ in $\R^3$. Here we think of $\R^3$ as the Lie
algebra of $SU(2)$ equipped with an invariant metric and the action
of $SO(3)$ on $TS^2$ is induced from the standard action of $SO(3)$ 
on~$\R^3$. Similarly $\PP^1 \times \PP^1$ is equivariantly homeomorphic
to $S^2 \times S^2$ equipped with the standard diagonal action of
$SO(3)$. 

Thus, to prove the claim it suffices to
exhibit an $SO(3)$-equivariant homeomorphism between 
$\PP(TS^2 \oplus \epsilon)$ and $S^2 \times S^2$. Denote by $\Delta$
the diagonal in $S^2 \times S^2$ and recall that the space
$\PP(TS^2 \oplus \epsilon) \smallsetminus \PP(TS^2)$ can be equivariantly
identified with 
(the total space of) $TS^2$
(see diagram (\ref{thomdiagram}). For $p \in S^2$ denote by $\St_p: S^2
\smallsetminus \{p\} \to \langle p \rangle^{\perp}$ the usual stereographic
projection from $p$ to the real $2$-plane orthogonal to $p$. Consider the map 
\begin{equation}\label{eq:stereographic}
(S^2 \times S^2) \smallsetminus \Delta \to TS^2, \quad (p,q) \mapsto \bigl(p, \St_p(q)\bigr)
\end{equation}
where we think of $TS^2$ as pairs $(x,y)$ in $\R^3$ such that $x \in
S^2$ and $y$ lies in the plane orthogonal to $x$. 
For $g \in SO(3)$, note that the usual stereographic projection
satisfies $\St_{gp}(gq) = g \St_p(q)$, so the map~\eqref{eq:stereographic} is
$SO(3)$-equivariant, and since each $\St_p$ is a homeomorphism, it
follows that~\eqref{eq:stereographic} is also a homeomorphism.

The space $\PP(TS^2 \oplus \epsilon)$ is obtained from the bundle $TS^2$ by
taking the one-point compactification of each fiber. Denote by
$\infty_p$ the point at infinity in the fiber of $\PP(TS^2 \oplus
\epsilon)$ over~$p$. The $SO(3)$ action  fixes each $\infty_p$,
and it is straightforward to 
check that the map~\eqref{eq:stereographic} can be equivariantly and
continuously extended to a map
\[
S^2 \times S^2 \to \PP(TS^2 \oplus \epsilon) 
\]
by defining $(p,p) \mapsto (p, \infty_p)$ for each $p \in S^2$. This
extension is evidently a bijection. A continuous bijection from a
compact space to a Hausdorff space is a homeomorphism, so we have constructed
the desired $G$-homeomorphism.
The two  properties of this homeomorphism specified in the last
sentence of the statement of the lemma follow from
the construction. 
\end{proof} 

We now define a $G$-equivariant continuous map $\Phi_{2r}:
(\PP^1)^{2r} \to F_{2r}$ as follows. First, recall that from 
Proposition~\ref{prop:filtration quotient as Thom space} we know that $F_2/F_0 \cong F_2$ is 
$G$-equivariantly homeomorphic to the Thom space
$\Thom\bigl(\tau\bigr)$. Second, we know $\PP(\tau \oplus
\epsilon)/\PP(\tau) \cong_G \Thom(\tau)$ by
Proposition~\ref{p:atiyah}. Putting these together with
the equivariant homeomorphism $\PP^1 \times \PP^1 \cong_G \PP(\tau
\oplus \epsilon)$ given in Lemma~\ref{lemma:PtauplusEpsilon-prodP1s}
we obtain the map $\Phi_{2r}$. More precisely, we have the following. 

\begin{definition}\label{definition:Phi 2r}
Let $\Phi_{2r}: (\PP^1)^{2r} \cong (\PP^1 \times \PP^1)^r \to F_{2r}$ denote the composition
\begin{equation}\label{eq:definition Phi 2r}
(\PP^1 \times \PP^1)^{r} \rTo
\bigl(\PP(\tau\oplus\epsilon)\bigr)^r\rTo \Thom(\tau)^r \rTo
  (F_2)^r\rTo F_{2r}
\end{equation}
where the first map is (the $r$-fold product of) the map $\Theta^{-1}$ from Lemma~\ref{lemma:PtauplusEpsilon-prodP1s},
the second is the ($r$-fold product of) the quotient map $\PP(\tau
\oplus \epsilon) \to \PP(\tau \oplus \epsilon)/\PP(\tau)$ together
with the identification in Proposition~\ref{p:atiyah}, the third is
the ($r$-fold product of the) map from Proposition~\ref{prop:filtration quotient as Thom space},
and the final arrow is induced from matrix multiplication as in~\eqref{matmult}. 
\end{definition} 

Since all the maps in the definition are continuous, the composition
$\Phi_{2r}$ is also continuous. The last arrow is $G$-equivariant
since the action of $G$ on $\Omega_{\poly}G$ is by conjugation and the
map is induced by multiplication. All the other maps are already known
to be equivariant, so the composition $\Phi_{2r}$ is also
$G$-equivariant.

We now turn our attention to a computation of $K^*_T((\PP^1)^n)$
(respectively $K^*_G((\PP^1)^n)$) and $H^*_T((\PP^1)^n;\Q)$ (respectively
$H^*_G((\PP^1)^n;\Q)$) for any positive integer $n$. We accomplish this
by an inductive argument that computes $K^*_T(X \times \PP^1)$
(respectively $H^*_T(X \times \PP^1;\Q)$) in terms of~$K^*_T(X)$
(respectively $H^*_T(X;\Q)$) for a general $G$-space $X$, using
equivariant Bott periodicity. Computations of $K^*_G$ and $H^*_G$ then
follow by taking Weyl invariants. 

Let $X$ be a $G$-space. By restricting the action, $X$ is then also a
$T$-space. 
Recall that $b \in R(T) \cong K_T(\pt)$ denotes the
standard $1$-dimensional representation of $T=S^1$ of weight $1$ and
that 
$v$ denotes the standard representation of $G=SU(2)$ on $\C^2$. As a
$T$-representation, $v \cong b \oplus b^{-1}$. 
Recall also that 
any $T$-representation
$\rho$ may also be viewed as an $T$-equivariant bundle over
a point~$\pt$. Let $\rho_X$ denote the pullback of $\rho$ to $X$ via
the constant map $X \to \pt$. This is a topologically trivial, but
equivariantly non-trivial, bundle over $X$. (Indeed, it is this map
$R(T) \to K^*_T(X), \rho \mapsto \rho_X$, which defines the $R(T)$-module
structure on $K^*_T(X)$. By slight abuse of notation, when considering
$\rho_X$ as an element in $K^*_T(X)$, we will sometimes denote it
simply by $\rho$.) 
In particular, the total space $E(\rho_X)$ of $\rho_X$ is
simply $X \times E(\rho)$ equipped with the diagonal $T$-action. Similarly, 
$\PP(\rho_X)\cong_T X \times \PP(\rho)$. 
In the case $\rho=b\oplus b^{-1}$, we identify 
$\PP\bigl((b \oplus b^{-1})_X\bigr)$ as a $T$-space
with $X \times \PP^1$ where the action of $T$ on $\PP^1$ is the standard one. 
Since $[zx:z^{-1}y]=[z^{k+1}x:z^{k-1}y]\in\PP^1$ for any $z \in T$, we have 
\begin{equation}\label{eq:Pb2 plus epsilon is P1}
\PP(b^2 \oplus \epsilon)\cong\PP(b\oplus b^{-1})\cong\PP^1
\end{equation}
as $T$-spaces.

We now apply the equivariant Bott periodicity results discussed in
Section~\ref{subsection:BottP} 
to the case where $H=T$ and  
$\zzeta$ is the bundle ~$b^2_X=b_X \otimes b_X$ over $X$, making use of the fact that
$X\times\PP^1\cong_T \PP(b_X^2\oplus\epsilon)$.
In the following,
 we denote by $b$ the image of $b \in R(T)$ under the map $R(T)
\to K^*_T\bigl(\PP(b_{\PP(b^2_X \oplus \epsilon)}^2 \oplus
\epsilon)\bigr)$. 
For the discussion it is useful to define the bundle 
\begin{equation}\label{def:calL}
\calL_X:=b^{-1}\otimes
\beta_{b^2_X \oplus \epsilon}
\end{equation}
over $\PP(b^2_X \oplus \epsilon)$.

We begin with the following.

\begin{lemma} 
The bundle $\calL_X$ of~\eqref{def:calL} has the property
\begin{equation}\label{Lequationmod}
\calL_X^2\oplus\epsilon\cong (b\otimes\calL_X)\oplus (b^{-1}\otimes\calL_X).
\end{equation}
In $K^*_T\bigl(\PP(b_X^2 \oplus \epsilon)\bigr)$, the (equivariant
$K$-theory class of the) bundle $\calL_X$ satisfies the relation 
\begin{equation}\label{Lequation}
(\calL_X -b)(\calL_X -b^{-1}) = 0
\end{equation}
\end{lemma} 

\begin{proof} 
Applying~\eqref{mainequationmod} to our situation yields
$b^2                              \oplus \epsilon \cong (b^2
 \otimes \beta_{b^2_X \oplus \epsilon}^{-1})
\oplus \beta_{b^2_X \oplus \epsilon}$. Tensoring with
$b^{-2} \otimes \beta_{b^2_X \oplus
  \epsilon}$ yields~\eqref{Lequationmod}.
Theorem~\ref{Bott} also states that 
$$(\beta_{b^2_X \oplus \epsilon} - b)(\beta_{b^2_X \oplus \epsilon} -1) =0$$
in $K_T\bigl(\PP(b_X^2 \oplus \epsilon)\bigr)$.
Dividing by~$b^2$ gives~\eqref{Lequation}.
\end{proof}

\begin{remark}
Dividing~\eqref{Lequation}
by $\calL_X^2$ shows that $\calL_X^{-1}$ satisfies the same relation.
\end{remark}

We next prove that the bundle $\calL_X$, considered there as a
$T$-bundle, is in fact naturally a $G$-bundle. In particular, the
$T$-equivariant $K$-theory class of $\calL_X$ is in the image of the
forgetful map $K^*_G(\PP(b^2_X \oplus \epsilon)) \to K^*_T(\PP(b^2_X
\oplus \epsilon))$. This observation will be useful when we compute
$K^*_G(\PP(b^2_X \oplus \epsilon))$ using the results for
$K^*_T(\PP(b^2_X \oplus \epsilon))$ (see Lemma~\ref{lemma:equivariant Bott} below).

\begin{lemma}\label{relatingLtogamma}
Let $\pi: \PP(b^2_X \oplus \epsilon) \to \PP(b^2 \oplus \epsilon) 
\cong \PP^1$ denote the natural $G$-equivariant projection and let
$\gamma$ denote the tautological line bundle over $\PP^1$. Then
$\calL_X \cong \pi^*(\gamma^{-1})$. In particular, since $\gamma$ (and
hence $\gamma^{-1}$) is a $G$-bundle, $\calL_X$ is also a $G$-bundle. 
\end{lemma}

\begin{proof}
From the defining equation~\eqref{basemainequation} of $\beta_{b^2_X
    \oplus \epsilon}$ together with the definition~\eqref{def:calL} of
  $\calL_X$, it can be deduced that 
  \begin{equation}
    \label{eq:b and gamma and calLX}
    b \oplus b^{-1}
       \cong \left(b^{-1} \otimes
    \gamma_{b^2_X \oplus \epsilon}\right) \oplus \calL_X.
  \end{equation}
Consider the special case $X=\pt$ where $\PP(b^2_X \oplus \epsilon) =
\PP(b^2 \oplus \epsilon) \cong \PP(b \oplus b^{-1}) \cong \PP^1$ as in~\eqref{eq:Pb2 plus epsilon
  is P1}. In addition, since the $G$-representation $v$ restricts to
the $T$-representation $b \oplus b^{-1}$, we may identify $\PP(b^2
\oplus \epsilon)$ with $\PP(v)$. In particular, it is a $G$-space.
Moreover, in this setting the $T$-bundle isomorphism~\eqref{eq:b and
  gamma and calLX} is the restriction to $T$ of the $G$-bundle
isomorphism 
\[
v_{\PP(v)} \cong \gamma_{v} \oplus \beta_{v}
\]
where the notation is as in Lemma~\ref{lemma:atiyah}
and $\gamma_{v} \cong \gamma$ (under the identification
$\PP(v) \cong \PP^1$). From this it follows that the bundle
$\calL_{\pt}$ is equivariantly isomorphic to the $G$-bundle
$\beta_{v}$. A Chern class computation shows that
$\beta_{v} \cong \gamma^{-1}$, so $\calL_{\pt} \cong
\gamma^{-1}$. This proves the result for the case $X=\pt$. On the
other hand, the bundle $\calL_X$ for a general space $X$ is $\calL_X =
\pi^*(\calL_{\pt}) \cong \pi^*(\gamma^{-1})$, as desired. 
\end{proof}

Motivated by the above lemma, we formulate below a version of
Theorem~\ref{Bott} for our situation, using the variable $\calL_X$
instead of $\beta_{b^2_X \oplus \epsilon}$. 

\begin{lemma}\label{lemma:equivariant BottT}
Let $X$ be a $G=SU(2)$-space (hence also a $T$-space). Let $T$ act on
$\PP^1$ in the standard way. Then 
\begin{equation}\label{BottT}
K^*_T\bigl(X\times\PP^1\bigr)
\cong \frac{K^*_T(X)[\calL_X]}{(\calL_X-b)(\calL_X-b^{-1})} \cong
\frac{K^*_T(X)[\calL_X]}{(\calL^2_X - (b+b^{-1})\calL_X +1)}.
\end{equation}
\end{lemma}

Lemma~\ref{lemma:equivariant BottT} computes $K^*_T(X \times \PP^1)$
in terms of $K^*_T(X)$, but we also need an analogous computation for
$K^*_G(X \times \PP^1)$. Recall that the forgetful
functor $F: K^*_G(Y) \to K^*_T(Y)$ for any $G$-space $Y$ has image
contained in the Weyl-invariants $K^*_T(Y)^W$, but $F$ need not be
surjective onto $K^*_T(Y)^W$ in general (i.e., it need not be the case
that $K^*_G(Y) \cong K^*_T(Y)^W$). See \cite[Example
4.8]{HarLanSja09}. Nevertheless, we have the following. 

\begin{lemma}\label{lemma:equivariant Bott}
  Let $X$ be a $G=SU(2)$-space. Let $G$ act on $\PP^1$ in the standard
  way. Assume that $K^*_G(X) \cong K^*_T(X)^W$. Then $K^*_G(X \times
  \PP^1) \cong K^*_T(X \times \PP^1)$, and 
\[
K^*_G(X \times \PP^1) \cong \frac{K^*_G(X)[\calL_X]}{(\calL^2_X - v
  \calL_X + 1)}.
\]
\end{lemma}

\begin{proof}
  The equivariant Thom isomorphism (Lemma~\ref{thomisom}) implies that
  $K^*_G(X)$ injects into $K^*_G(X \times \PP^1)$. We showed in
  Lemma~\ref{relatingLtogamma} that $\calL_X$ lies in $K^*_G(X \times
  \PP^1 \cong \PP(b^2_X \oplus \epsilon))$. 
If we make the assumption that 
$K^*_G(X) \cong K^*_T(X)^W$, 
taking Weyl invariants
  in~\eqref{BottT} gives
\[
K^*_T(X \times \PP^1)^W \cong \frac{K^*_G(X)[\calL_X]}{(\calL_X^2 -
  (b+b^{-1})\calL_X + 1)}
\]
since $\calL_X$ and $(b+b^{-1}) \cong v$ are $W$-invariant.
 Hence $K^*_T(X \times
\PP^1)^W$ is generated by $K^*_G(X)$ and $\calL_X$, i.e., the
restriction map $$K^*_G(X \times \PP^1) \to K^*_T(X \times \PP^1)$$ 
is surjective and $K^*_G(X \times \PP^1) \cong K^*_T(X \times \PP^1)^W$.
The result follows. 
\end{proof}

Next we record the analogous computations in cohomology. Since the
ideas are similar we keep exposition brief. 
Let $\bar{b}=c_1^T(b)\in H_T^2(\pt)=H^2(BT;\Z)$ denote the equivariant Chern
class of~$b \in R(T)$.
Then $H_T^*(\pt;\Z)=\ZZ[\bar{b}]$ and $c_1^T(b^{-1})=-\bar{b}$.
The nontrivial element $w$ of the Weyl group $W \cong S_2 \cong \Z/2\Z$ of $G$ acts
by $w(\bar{b})=-\bar{b}$
and $$H_G^*(\pt;\Z)=H^*(BG;\Z)=\bigl(H^*(BT;\Z)\bigr)^W=\Z[\bar{t}],$$
where 
\begin{equation}\label{def bart}
\bar{t}:=\bar{b}^2. 
\end{equation}
As in the case of $K$-theory, there is a canonical ring homomorphism
$$H^*_T(\pt;\Z) = H^*(BT;\Z) \to H^*_T(X;\Z)$$ (respectively $H^*_G(\pt;\Z) \to
H^*_G(X;\Z)$) for any $T$-space (respectively $G$-space)~$X$, and by
slight abuse of notation we will use the same symbol to denote an
element in $H^*_T(\pt;\Z)$ (respectively $H^*_G(\pt;\Z)$) to denote its
image in $H^*_T(X;\Z)$ (respectively $H^*_G(X;\Z)$). 

Let $\bar{\calL}_X$ denote the $T$-equivariant first 
Chern class of $\calL_X$, i.e.,
$\bar{\calL}_X := c_1^T(b^{-1}_{\PP(b^2_X \oplus \epsilon)} \otimes
\beta_{b^2_X \oplus \epsilon})$ in
$H_T^2\bigl(\PP(b_X^2\oplus\epsilon);\Q\bigr)$.

\begin{lemma}\label{lemma:barcalL squared} 
Let $\bar{\calL}_X$ and $\bar{t}$ be as above. Then $\bar{\calL}_X^2 = \bar{t}$. 
\end{lemma}

\begin{proof} 
The isomorphism~\eqref{Lequationmod} implies that $\calL_X^2 \oplus
\epsilon$ is a sum of a line bundle
with a trivial bundle. Hence its $2$nd equivariant Chern class is $0$
and 
we have 
\begin{equation}\label{Lsquared}
\begin{split}
0 & =c_2^T\bigl((b_{\PP(b^2_X \oplus \epsilon)} \oplus
b^{-1}_{\PP(b^2_X \oplus \epsilon)}) \otimes \calL_X\bigr) \\
 & = c_1^T(b_{\PP(b^2_X \oplus \epsilon)} \otimes
 \calL_X)c_1^T(b^{-1}_{\PP(b^2_X \oplus \epsilon)} \otimes \calL_X) \\
 & = (\bar{\calL}_X+\bar{b})(\bar{\calL}_X-\bar{b}) \\
 & =\bar{\calL}_X^2-\bar{b}^2 \\
 & =\bar{\calL}_X^2-\bar{t}
\end{split}
\end{equation}
as desired. 
\end{proof} 

\begin{remark} 
Lemma~\ref{lemma:barcalL squared}
is the cohomology analogue of~\eqref{Lequation}.
\end{remark} 

The equivariant Thom isomorphism in cohomology implies 
\begin{equation}\label{HTPequation}
\tilde{H}_T^*\bigl(\PP(b_X^2\oplus\epsilon);\Q\bigr)\cong
\bar{p}^*\bigl(H_T^*(X;\Q)\bigr)\oplus
\bar{U}\cdot\bar{p}^*\bigl(H_T^*(X;\Q)\bigr)
\end{equation}
as $H_T^*(\pt;\Q)$-modules,
where $\bar{U}:=\bar{U}_{b^2_X}:=c_1^T(U_{b^2_X})\in H_T^2\bigl(\PP(\tilde{b}^2\oplus\epsilon;\Q)\bigr)$
is the equivariant Thom class in cohomology. Recalling from
Lemma~\ref{lemma:compute Uzzeta}
that $U_{b^2_X} = 1 - \beta_{b^2_X \oplus \epsilon}$, we obtain 
$\bar{U}=-\bar{\calL}_X-\bar{b}$. Thus~\eqref{HTPequation} can be rewritten as
$$\tilde{H}_T^*\bigl(\PP({b}_X^2\oplus\epsilon);\Q\bigr)\cong
\bar{p}^*\bigl(\tilde{H}_T^*(X;\Q)\bigr)\oplus
\bar{\calL}_X\cdot\bar{p}^*\bigl(H_T^*(X;\Q)\bigr)$$
as $H^*_T(\pt;\Q)$-modules.
Together with $\bar{L}^2=\bar{t}$, this determines the $H_T^*(\pt;\Q)$-algebra
structure as follows. 

\begin{lemma}\label{lemma:equivariant Bott cohomology T}
With notation as above, 
$$H_T^*\bigl(\PP(\tilde{b}^2\oplus\epsilon);\Q\bigr)\cong
H_T^*(X;\Q)[\bar{\calL}_X]/\bigl(\bar{\calL}_X^2-\bar{t}\bigr).$$
\end{lemma} 

Taking Weyl invariants and using the fact that $H^*_G(X;\Q) \cong
H^*_T(X;\Q)^W$ for any $G$-space~$X$, we also conclude the
following.\footnote{We warn the reader that for cohomology with $\Z$ coefficients it is not
  always the case that $H^*_G(X;\Z) \cong H^*_T(X;\Z)^W$. See
  \cite{HolSja:2008}.}

\begin{lemma}\label{lemma:equivariant Bott cohomology} 
Let $X$ be a $G$-space and let $G$ act on $\PP^1$ in the standard way.
Then 
$$H_G^*\bigl(X\times\PP^1 \cong \PP(b^2_X \oplus \epsilon);\Q\bigr)\cong\Bigl(H_T^*\bigl(X\times\PP^1\bigr);\Q\Bigr)^W
\cong H_G^*(X;\Q)[\bar{\calL}_X]/\bigl(\bar{\calL}_X^2-\bar{t}\bigr)$$
as $H_G^*(\pt;\Q)$-algebras.
\end{lemma}

With the preceding lemmas in place, it is now straightforward to
compute $K^*_T((\PP^1)^n)$, $K^*_G((\PP^1)^n)$, $H^*_T((\PP^1)^n;\Q)$, and
$H^*_G((\PP^1)^n;\Q)$ by a simple inductive argument starting with $X =
\pt$. 
Let $\pi_j: (\PP^1)^n \to \PP^1$ denote the projection to the $j$-th coordinate.
Motivated by Lemma~\ref{relatingLtogamma}, 
we define the following collection of line bundles over $(\PP^1)^n$. 

\begin{definition}\label{definition:L_j and barL_j}
For any $j$ with $1 \leq j \leq n$, define 
\[
L_j := \begin{cases} \pi_j^*(\gamma^{-1}) \textup{  if $j$ is odd}; \\
\pi_j^*(\gamma) \textup{  if $j$ is even.} \\
\end{cases} 
\]
We denote by $\bar{L}_j$ the first Chern class $c_1(L_j)$ of $L_j$ in 
$H_T^2\bigl((\PP^1)^n;\Q \bigr)$.
\end{definition}

We remark that all the bundles $L_j$ in the above definition are
equipped with a natural $G$-action (and hence also a $T$-action). In
particular they may be viewed as elements of
$K^*_T\bigl((\PP^1)^n\bigr)$ or $K^*_G\bigl((\PP^1)^n\bigr)$.

Using Lemmas~\ref{lemma:equivariant BottT}, ~\ref{lemma:equivariant
  Bott}, ~\ref{lemma:equivariant Bott cohomology T}
and~\ref{lemma:equivariant Bott cohomology} and 
and a simple induction argument beginning with $X=\pt$ we
obtain the following. 

\begin{theorem}\label{theorem:P1summary}
Let $G=SU(2)$ act on $\PP^1$ in the standard way.
Let $L_j$ be the bundles defined above, viewed as
  elements of $K^*_T \bigl((\PP^1)^n\bigr)$ or $K^*_G
  \bigl((\PP^1)^n\bigr)$ as appropriate, and let $\bar{L}_j$ denote
  the first Chern classes of $L_j$, viewed as elements of $H^2_T((\PP^1)^n;\Q)$
  or $H^2_G((\PP^1)^n;\Q)$ as appropriate. 
Then 
\begin{equation}
    K^*_T\bigl((\PP^1)^n\bigr) \cong
    R(T)[L_1,\ldots,L_n]/\langle\{L_j^2-(b+b^{-1})L_j+1, \, 1 \leq j \leq
    n\}\rangle
 \end{equation} 
and 
\begin{equation} K^*_G\bigl((\PP^1)^n\bigr) \cong
    R(G)[L_1,\ldots,L_n]/\langle\{L_j^2-vL_j+1, \, 1 \leq j \leq n\}\rangle.
  \end{equation} 
We also have 
\begin{equation}
H_T^*\bigl((\PP^1)^n;\Q\bigr)
\cong\Z[\bar{b}][\bar{L}_1,\ldots,\bar{L}_n]
/\langle\{\bar{L}_j^2-\bar{b}^2, \, 1 \leq j \leq n \}\rangle
\end{equation}
and
\begin{equation}
H_G^*\bigl((\PP^1)^n;\Q\bigr)
\cong\Z[\bar{t}][\bar{L}_1,\ldots,\bar{L}_n]
/\langle\{\bar{L}_j^2-\bar{t}, \, 1 \leq j \leq n\}\rangle.
\end{equation}
\end{theorem}

\begin{remark}
The relation $L_j^2 - v L_j +1=0$ appearing in the theorem above is the pullback to the $j$th factor
of the relation in Remark \ref{rem:Milnor}.
\end{remark}

\section{Computation of $\Phi_{2r}^*: H^*_G(F_{2r};\Q) \to
  H^*_G((\PP^1)^{2r};\Q)$}\label{sec:computation cohomology}

The main goal of this section is to give explicit descriptions of
$H^*_G(F_{2r};\Q)$ by using the map $\Phi_{2r}^*: H^*_G(F_{2r};\Q)
\to H^*_G((\PP^1)^{2r};\Q)$ and using the presentation of the codomain
given in Theorem~\ref{theorem:P1summary}. Specifically, we show in
Theorem~\ref{HTF2rQ}
that 
$$\Phi_{2r}^*: H^*_T(F_{2r};\Q) \to H^*_T((\PP^1)^{2r};\Q)$$ is
injective and concretely describe the subring
$\Phi_{2r}^*(H^*_T(F_{2r};\Q))$. In Theorem~\ref{HFG2rQ} we take Weyl
invariants to obtain the corresponding result in $H^*_G$. We will need
the cohomology results in the present section in order to complete the
analogous computation for $K$-theory in Section~\ref{section:ImPhiK}.

Consider the commutative diagram
\begin{equation}\label{F2diagram}
\begin{diagram}
\pt &\rTo^{i_1\circ i_1}& \PP^1 \times \PP^1 \cr
\dEqualto &&  \dTo_{\Phi_2}\cr
\pt &\rTo& F_2 \cr
\end{diagram}
\end{equation}
where $i_1 \circ i_1(\pt)$ is by definition the point $([T],[T]) \in
\PP^1 \times \PP^1 \cong G/T \times G/T$, and the bottom horizontal
arrow has image the `point at infinity' in the Thom space $F_2 \cong
\Thom(\tau)$. 

\begin{remark}\label{rk1}
Note that $i_1 \circ i_1$ is not a $G$-equivariant map, since the
image point $$i_1 \circ i_1(\pt) = ([T],[T])$$ is not a $G$-fixed point.
All other maps in the diagram are $G$-equivariant. 
\end{remark}

Recall that since $G=SU(2)$ is
simply-connected and hence
$$H^1(BG)=H^2(BG)=0,$$ the Serre spectral
sequence for~$H^*_G$ implies that $H^2_G(X) \cong H^2(X)$ for any
$G$-space $X$. Consider the diagram~\ref{maindiagram} for the special
case $\zzeta=\tau$ and $X=\PP^1$. In this setting, the map $\PP(\tau
\oplus \epsilon) \to \Thom(\tau) \cong F_2$ is the composition $\Phi_2
\circ \Theta$, where $\Theta: \PP(\tau \oplus \epsilon) \to \PP^1 \times \PP^1$ is the equivariant homeomorphism constructed in Lemma~\ref{lemma:PtauplusEpsilon-prodP1s}. (To be completely precise, we should include in the
notation the homeomorphism $\Thom(\tau) \cong F_2$ mentioned in
Proposition~\ref{prop:filtration quotient as Thom space} and which is
part of the definition of the map $\Phi_2$, but since this map plays
no role in the argument we will ignore this ambiguity.) The
fact that $j$ splits in diagram~\ref{maindiagram} implies
that $(\Phi_2 \circ \Theta)^*:
H^2_G(F_2;\Z) \cong H^2(F_2;\Z) \to H^2_G(\PP(\tau \oplus
\epsilon);\Z) \cong H^2(\PP(\tau \oplus \epsilon);\Z)$ is an
injection. Let $\bar{U}_{\tau} \in \Ker j^* \subset H^2_G(\PP(\tau \oplus \epsilon);\Z)
\cong H^2(\PP(\tau \oplus \epsilon);\Z)$ denote the (equivariant)
cohomology Thom class of the bundle $\tau$ over $\PP^1$. 
With these preliminaries in place we may now define the following. 

\begin{definition}\label{definition:L over F2}
We define $L$ as the unique line bundle over~$F_{2}$ 
such that $$c_1((\Phi_2 \circ \Theta)^*(L)) = (\Phi_2 \circ
\Theta)^*(c_1(L)) = \bar{U}_\tau.$$ 
\end{definition}

The next lemma is fundamental; it shows that the pullback of $L$ is
symmetric in $L_1$ and $L_2$ on $\PP^1 \times \PP^1$ (where the $L_i$
are as in Definition~\ref{definition:L_j and barL_j}). 

\begin{lemma}\label{PhiL}
Let $\Phi_2: \PP^1 \times \PP^1 \to F_2$ be the map given in
Definition~\ref{definition:Phi 2r} and $L, L_1, L_2$ be the line bundles 
given in Definitions~\ref{definition:L over F2}
and~\ref{definition:L_j and barL_j}. Then 
\[
\Phi_2^*(L) \cong L_1\otimes L_2. 
\]
\end{lemma}

\begin{remark}
The definition of the $L_j$ given in Definition~\ref{definition:L_j
  and barL_j}, which treated differently the cases when $j$ is odd and
$j$ is even, is motivated by Lemma~\ref{PhiL}. If we define $L_j$ by
the same formula for all~$j$, then we lose the symmetry in the
statement of Lemma~\ref{PhiL}. 
\end{remark}

\begin{proof} 
As noted above, $H^2_G(\PP^1 \times \PP^1; \Z) \cong H^2(\PP^1 \times
\PP^1;\Z) \cong \Z \times \Z$ since $G$ is simply-connected, and $H^2(\PP^1 \times
\PP^1)$ is generated by $c_1(L_1)$ and $c_1(L_2)$ by definition of the
$L_i$. Hence $c_1(\Phi_2^*(L)) = m \cdot c_1(L_1) + n \cdot c_1(L_2)$ for some
integers $m$ and $n$. Equivalently, $$\Phi_2^*(L) \cong L_1^{\otimes m}
\otimes L_2^{\otimes n}.$$ 
Moreover, since $\Theta$ is a $G$-equivariant homeomorphism, this is
equivalent to 
\[
c_1((\Phi_2 \circ \Theta)^*(L)) = m \cdot c_1(\Theta^*(L_1)) + n \cdot
c_1(\Theta^*(L_2)).
\]
We first give a different description of the left hand side of this
equation. Since $\tau$ is a line bundle, Lemma~\ref{lemma:compute
  Uzzeta} implies that $U_\tau = 1 - \beta_{\tau \oplus \epsilon} \in
\Ker j^* \subset K^*_G(\PP(\tau \oplus \epsilon))$ which in turn means
\(\bar{U}_\tau = - c_1(\beta_{\tau \oplus \epsilon}) \in
H^2_G(\PP(\tau \oplus \epsilon)) \cong H^2(\PP(\tau \oplus
\epsilon)).\) Thus, by definition $$c_1((\Phi_2 \circ \Theta)^*(L)) = -
c_1(\beta_{\tau \oplus \epsilon}),$$ so in fact it suffices to compute
the coefficients $a,b \in \Z$ in the equality 
\[
c_1(\beta_{\tau \oplus \epsilon}) = a \cdot c_1 (\Theta^*(L_1)) + b
\cdot 
c_1(\Theta^*(L_2)).
\]
By Lemma~\ref{lemma:PtauplusEpsilon-prodP1s}, the composition $\Theta
\circ k: \PP^1 \to \PP(\tau \oplus \epsilon) \to \PP^1 \times \PP^1$
(where $k$ is the inclusion of the fiber in~\eqref{maindiagram})
is the inclusion of $\PP^1$ as the second factor. This implies
$k^*\Theta^*L_1$ is the trivial bundle, so $k^*c_1(\Theta^*L_1) = 0$ 
and $k^*c_1(\beta_{\tau\oplus \epsilon}) = b k^*c_1(\Theta^*L_2) = b
c_1(L_2)$. Next recall that the defining equation of $\beta_{\tau
  \oplus \epsilon}$ is 
\[
\gamma_{\tau \oplus \epsilon} \oplus \beta_{\tau \oplus \epsilon} =
\bar{p}^*(\tau \oplus \epsilon)
\]
so in particular $k^*(\gamma_{\tau \oplus \epsilon} \oplus \beta_{\tau
  \oplus \epsilon})$, being the restriction of $\bar{p}^*(\tau \oplus
\epsilon)$ to a fiber of $\bar{p}$, is the trivial bundle. Hence 
\[
c_1(k^*\gamma_{\tau \oplus \epsilon}) = k^*c_1(\gamma_{\tau \oplus
  \epsilon}) = - k^*c_1(\beta_{\tau \oplus \epsilon}).
\]
On the other hand, $k^*(\gamma_{\tau \oplus \epsilon}) \cong \gamma$
by definition of the tautological bundle, so $c_1(k^*\gamma_{\tau
  \oplus \epsilon}) =  c_1(L_2)$ by definition of $L_2$. Hence
$k^*c_1(\beta_{\tau \oplus \epsilon}) = - c_1(L_2)$ and we conclude
$b=-1$ and hence $n=1$. 

It remains to compute the coefficient $m$. Again by
Lemma~\ref{lemma:PtauplusEpsilon-prodP1s}, the composition $\Theta
\circ j: \PP^1 \cong \PP(\tau) \to \PP(\tau \oplus \epsilon) \to
\PP^1\times \PP^1$ is the inclusion of $\PP^1$ as the diagonal in the
direct product. This implies that $(\Theta \circ j)^*(L_1) =
\gamma^{-1}$ and $(\Theta \circ j)^*(L_2) = \gamma$, so $(\Theta \circ
j)^*(L_1^m \otimes L_2^n) = 1$ if and only if $n=m$. 
Thus $n=m=1$, as desired. 

\end{proof}

The following is immediate from Lemma~\ref{PhiL}. Recall that we also denote by $\bar{L}_i$
the Chern class $c_1^T(L_i)$ of the $L_i$. 

\begin{lemma}
Under  $\Phi_2^*: H^*_T(F_2;\Z) \to H^*_T(\PP^1 \times
\PP^1)$, the Chern class $\bar{L} := c_1^T(L) \in H^2_T(F_2;\Z)$ maps
to $\bar{L}_1 + \bar{L}_2$. 
\end{lemma}

From \cite{HJS12} it follows
that the inclusion $F_{2r-2} \to F_{2r}$ induces isomorphisms on
$$H_G^2(F_{2r};\Z) \cong \Z$$ for all $r>0$. By slight abuse of notation
we define the line bundle $L$ on $F_{2r}$ to be the unique bundle such
that the composite inclusion $F_2 \to F_4 \to \cdots \to
F_{2r}$ pulls $L$ on $F_{2r}$ back to the bundle $L$ over $F_2$ given in
Definition~\ref{definition:L over F2}. Since the inclusion is a
$T$-equivariant map and the bundle $L$ of Definition~\ref{definition:L
  over F2} is a $T$-bundle, the bundle $L$ over $F_{2r}$ is also a
$T$-bundle. 

A straightforward generalization of the argument given for
Lemma~\ref{PhiL} yields the following.

\begin{lemma}\label{lemma:s1 in image}
Let $\Phi_{2r}: (\PP^1)^{2r} \to F_{2r}$ be the map given in
Definition~\ref{definition:Phi 2r} and let $L$ be the line bundle over
$F_{2r}$ defined above. Let $L_i, 1 \leq i \leq 2r$, be the line
bundles given in Definition~\ref{definition:L_j and barL_j}. Then
\[
\Phi_{2r}^*(L) \cong L_1 \otimes \cdots \otimes L_{2r}.
\]
In particular, 
\[
\Phi_{2r}^*(c_1^T(L)) = \bar{L}_1 + \cdots + \bar{L}_{2r} \in
H^2_T((\PP^1)^{2r};\Q).
\]
\end{lemma}

We are now ready to prove the main results of this section. 
Let $s_j(y_1,\ldots,y_n)$ denote the $j$th elementary symmetric polynomial
in the variables~$y_1,\ldots,y_n$, where by standard convention $s_0:=1$.
We let $\bar{s}_j$ denote the element 
\begin{equation}\label{eq:def barsj}
\bar{s}_j(\bar{L}_1,\ldots, \bar{L}_{2r})\in H^*_G\bigl((\PP^1)^{2r}\bigr)
\subset H^*_T\bigl((\PP^1)^{2r}\bigr). 
\end{equation}
The next result proves that the image of $\Phi_{2r}^*:
H^*_T(F_{2r};\Q) \to H^*_T((\PP^1)^{2r};\Q)$ is generated precisely by
the $\bar{s}_j$, i.e., it is the subalgebra of symmetric polynomials
in the elements $\bar{L}_i$ with respect to the presentation of
$H^*_T((\PP^1)^{2r};\Q)$ given in Theorem~\ref{theorem:P1summary}. 

\begin{theorem}\label{HTF2rQ}
Let $G=SU(2)$ and $T \subset G$ be the standard maximal torus. Let
$\Phi_{2r}: (\PP^1)^{2r} \to F_{2r}$ be the map given in
Definition~\ref{definition:Phi 2r}. Then:
\begin{enumerate} 
\item The pullback $\Phi_{2r}^*: H^*_T(F_{2r};\Q) \to H^*_T\bigl((\PP^1)^{2r};\Q\bigr)$ is
  injective. 
\item The ring $H^*_T(F_{2r};\Q)$, or equivalently the image of $\Phi_{2r}^*: H^*_T(F_{2r};\Q) \to
  H^*_T((\PP^1)^{2r};\Q)$ is the symmetric subalgebra of
  $H^*_T((\PP^1)^{2r};\Q)$. More specifically,  
\begin{equation*}
\begin{split} 
H_T^*(F_{2r};\Q) & \cong\Phi_{2r}^*\bigl(H_T^*(F_{2r};\Q)\bigr) \\ 
&= \bigl(H_T^*(\PP^1)^{2r};\Q\bigr)^{S_{2r}} \\
&=\{\mbox{symmetric polynomials
  in~$H_T^*\bigl((\PP^1)^{2r};\Q\bigr)$}\} \\ 
&= \mbox{the $H^*_T(\pt;\Q)$-subalgebra
of $H_T^*\bigl((\PP^1)^{2r};\Q\bigr)$ generated
by $\bar{s}_1,\ldots,\bar{s}_{2r}$} \\ 
&= \mbox{the $H^*_T(\pt;\Q)$-submodule
of $H_T^*\bigl((\PP^1)^{2r};\Q\bigr)$ generated
by $\bar{s}_0,\ldots,\bar{s}_{2r}$}. \\
\end{split} 
\end{equation*}
\end{enumerate}
\end{theorem}

For the proof of Theorem~\ref{HTF2rQ} we need the following notation. 
Let $\Q[\bar{t}]$ be a polynomial ring in one variable $\bar{t}$ and let 
$A=\Q[\bar{t},\bar{L}_1,\ldots, \bar{L}_N] = \Q[t][\bar{L}_1,\ldots, \bar{L}_N]$ denote a polynomial
ring in the variable $\bar{t}$ of degree $2$ and the variables $\bar{L}_1, \ldots,
\bar{L}_{2r}$, each of degree $1$. Then $A$ is also a $\Q[t]$-module. (The cohomology degree of the
corresponding cohomology classes are obtained by multiplying the
polynomial degree by $2$.) 
Let $I$ denote the ideal in $A$ generated by 
the homogeneous polynomials $\{\bar{L}_j^2 - \bar{t}, 1 \leq j \leq N\}$. Then by
Theorem~\ref{theorem:P1summary}, we know the ring
$H^*_T\bigl((\PP^1)^N;\Q\bigr)$ can be presented as the quotient
ring $A/I$. We will use this identification in the proof below.

Further, for a monomial $\bar{t}^b\bar{L}^\alpha$ where $\alpha=(\alpha_1,\ldots,
\alpha_{2r})$, we define its
{\em signature} by
$$\sig(\bar{t}^b \bar{L}^\alpha):=\sig(\alpha):= 
\# \textup{$j$ such that  $\alpha_j$ is odd} .$$
By definition, the signature takes values in $\{0, 1, \ldots, 2r\}$
and induces a $\Q[\bar{t}]$-module decomposition $\oplus_{s=0}^{2r}
A_{\sig(\alpha)=s}$ of $A$, where $A_{\sig(\alpha)=s}$ consists of the
$\Q[\bar{t}]$-submodule of $A$ generated by the monomials of signature $s$. 
Notice that $I$ is homogeneous with respect to this grading in the sense that $I =
\oplus_{s=0}^{2r} (I \cap A_{\sig(\alpha)=s})$ since each monomial of the generators $\bar{L}^2_j - \bar{t}$ are of signature $0$.  
Finally, for a fixed degree $M>0$ and for a multi-index $\alpha =
(\alpha_1, \ldots, \alpha_{2r}),$ we set 
$$b(\alpha) := \frac{M - (\alpha_1 + \ldots + \alpha_{2r})}{2} $$
so that the monomial $\bar{t}^{b(\alpha)}\bar{L}^\alpha$ has total
degree~$M$.
Similarly we set 
\[
b'(\alpha) = \frac{M-2 - (\alpha_1 + \ldots + \alpha_{2r})}{2} =
b(\alpha)-1 
\]
so that the monomial $\bar{t}^{b'(\alpha)}\bar{L}^\alpha$ has total
degree $M-2$.

\begin{proof} 

First we claim that the $H^*_T(\pt)$-subalgebra generated
by $\bar{s}_1,\ldots,\bar{s}_{2r}$ is contained in ${\im} \Phi_{2r}^*$.
Recall that, working over $\Q$, the monomials in the power functions 
$\bar{\psi}^k:=\bar{L}_1^k+\ldots +\bar{L}_{2r}^k$ form a basis for all symmetric polynomials in $\bar{L}_j$. 
Using the relations $\bar{L}^2_i - \bar{t}$, this means we can express the 
(equivalence class of) any symmetric polynomial in the $\bar{L}_j$ as a 
polynomial in $\bar{s}_1$. We saw in Lemma~\ref{lemma:s1 in image}
above that $\bar{s}_1$ is in the image of $\Phi_{2r}^*$, so the claim
follows.

Note that we have the natural inclusions 
\begin{equation}\label{eq:inclusions} 
\mbox{$\Q[t]$-module generated by  $ \bar{s}_0,\ldots,\bar{s}_{2r}$}
\subseteq
\mbox{$\Q[t]$-algebra generated by  $ \bar{s}_1,\ldots,\bar{s}_{2r}$}
 \subseteq \im \Phi_{2r}^*. 
\end{equation}
We show next that the $\Q[t]$-submodule generated by $\bar{s}_0, \ldots,
\bar{s}_{2r}$ is a free $\Q[t]$-module of rank
$2r+1$.

Since $I$ is a homogeneous ideal in the usual grading on $A$, in order
to prove the claim, it suffices to prove that if 
(the equivalence class of) a linear combination $y = \sum_{j = 0}^{2r} n_j t^{\frac{M-j}{2}} \bar{s}_j$ of
homogeneous degree $M$ is equal to $0$ in the quotient ring 
$A/I \cong H^*_T((\PP^1)^{2r};\Q)$, where each $n_j \in \Q$, then each
$n_j$ must be equal to $0$. (Here we use the convention that if
$\frac{M-j}{2}$ is not integral then $n_j$ is automatically equal to
$0$.) 
Suppose we have such a $y$. The equivalence class of $y$ is $0$ in $A/I$ if and only if $y \in I$,
so there exists a relation 
\begin{equation} \label{eq1}
y = \sum_{j=1}^N a_j(\bar{L}_j^2 -\bar{t}) 
\end{equation}
in $A$, where $a_j = a_j(\bar{t},\bar{L}_1, \ldots, \bar{L}_N) \in A$.

Fix an integer $k$ with $0 \leq k \leq 2r$. Assume $Q := \frac{M-k}{2}$ is
integral. We wish to show that $n_k=0$. 
Taking the component of~\eqref{eq1} lying in $A_{s(\alpha)=k}$ yields the
equality 
\begin{equation}\label{eq1A}
n_k \bar{t}^Q \bar{s}_k = \sum_{j=1}^{2r} a_{j,k}(\bar{L}_j^2 -\bar{t}),
\end{equation}
where $a_{j,k}$ denotes the component of $a_j$ lying in
$A_{s(\alpha)=k}$. 
Let 
$c_{j,b,\alpha} \in \Q$ denote the coefficient of 
$\bar{t}^b {\bar{L}_1}^{\alpha_1} \ldots {\bar{L}_{2r}}^{\alpha_{2r}} $ in the
polynomial~$a_{j,k}$. By convention we set $c_{j, b,
  \alpha} = 0 $ if any entry of $\alpha \in \Z^{2r}$ is negative
or if $b'(\alpha)$ is not integral. Note also that $b$ must satisfy $b = b'(\alpha)$ for
$c_{j,b,\alpha}$ to be non-zero since $a_{j,k}$ is homogeneous of
degree $M-2$. 

Let $m_{b,\alpha} $ denote the coefficient of the monomial $\bar{t}^b{\bar{L}_1}^{\alpha_1} \ldots {\bar{L}_{2r}}^{\alpha_{2r}}$
after collecting terms on the right hand side of~\eqref{eq1}. 
It follows that 
\begin{equation} \label{eq2}
m_{b,\alpha} = \sum_j c_{j,b,(\alpha_1, \ldots, \alpha_j -2, \ldots, \alpha_{2r})} + 
\sum_j c_{j,b-1,\alpha}. 
\end{equation}
Now define for any $b \in \Z$ the expression 
$$C_b := \sum_{\{\alpha~s.t.\ b'(\alpha) = b\}}
\sum_j c_{j,b,\alpha}. $$

Note that a given coefficient of the form $c_{j, b, \gamma}$ occurs in
the first summation of~\eqref{eq2} for $2r$ different values of the
multi-index $\gamma$. 
Thus
\begin{equation} \label{eq2A}
\sum_{\{\alpha~s.t.~ b(\alpha) = b\}} m_{b,\alpha} = (2r) \cdot C_b +
C_{b-1}. 
\end{equation}
Equating this to the corresponding sum on the left hand side of~\eqref{eq1A}
yields
$$(2r) \cdot C_b+C_{b-1}=
\begin{cases} 0 & {\rm if\ } b\ne Q; \cr
n_k {2r\choose k} & {\rm if\ } b = Q. \cr
\end{cases}
$$ 
Since $C_{-1}=0$ we inductively conclude that $C_b=0$ for $b \le Q-1$, and
in particular, $C_{Q-1}=0$.

We next claim that $C_Q =0$, 
To see this, recall $a_{j,k}$ is of degree
$M-2$ and each monomial~$t^b L^\alpha$ appearing with non-zero
coefficient in $a_{j,k}$ must also have
signature $k$. This implies that the degree of the
$L^\alpha$ must be $\geq k$ (and hence $b$ must be $\leq Q-1$), so any
coefficient of the form $c_{j, Q, \alpha}$ is $0$. Hence $C_Q=0$ and
we conclude 
\[
C_{Q-1} = n_k {2r \choose k} = 0
\]
from which it also immediately follows that for any $k$ with $0 \leq k
\leq 2r$, we also have $n_k=0$. 

We conclude that the $\{\bar{s}_j\}_{j=0}^{2r}$ generate a free
$\Q[t]$-module in $H^*_T((\PP^1)^{2r};\Q)$. Denote this
free module by $\mathcal{M}$. Let $\mathcal{M}^{2\ell}$ denote the (cohomology)
degree-$2\ell$ piece of $\mathcal{M}$. Since the (cohomology) degree of $\bar{s}_j$ is
$2j$, the dimension of $\mathcal{M}^{2\ell}$ as a $\Q$-vector space is $\ell+1$
for $0 \leq \ell \leq 2r$ and $2r+1$ for $\ell > 2r$. 
This agrees with the dimensions of $H_T^{2\ell}(F_{2r};\Q)$ for
all~$\ell$ \cite{HJS12}. 
Thus the containments in~\eqref{eq:inclusions}
must be equalities. In particular, $\im \Phi_{2r}^* \cong
H^*_T(F_{2r};\Q)$ and we conclude $\Phi_{2r}^*$ is injective. The rest
of the theorem also follows. 
\end{proof}

Taking Weyl invariants gives the following. 

\begin{theorem}\label{HFG2rQ}
Let $G=SU(2)$ and $T \subset G$ be the standard maximal torus. Let
$\Phi_{2r}: (\PP^1)^{2r} \to F_{2r}$ be the map given in
Definition~\ref{definition:Phi 2r}. Then:
\begin{enumerate} 
\item The pullback $\Phi_{2r}^*: H^*_G(F_{2r};\Q) \to H^*_G((\PP^1)^{2r};\Q)$ is
  injective. 
\item The ring $H^*_G(F_{2r};\Q)$, or equivalently the image of $\Phi_{2r}^*: H^*_G(F_{2r};\Q) \to
  H^*_G((\PP^1)^{2r};\Q)$ is the symmetric subalgebra of
  $H^*_G((\PP^1)^{2r};\Q)$. More specifically,  
\begin{equation*}
\begin{split} 
H^*_G(F_{2r};\Q) & \cong\Phi_{2r}^*\bigl(H_G^*(F_{2r};\Q)\bigr) \\
&= \bigl(H_G^*(\PP^1)^{2r};\Q\bigr)^{S_{2r}} \\
&=\{\mbox{symmetric polynomials in~$H_G^*\bigl(\PP^1)^{2r};\Q\bigr)$}\}\\
&= \mbox{the $H^*_G(\pt;\Q)$-subalgebra
of $H_G^*\bigl((\PP^1)^{2r};\Q\bigr)$ generated
by $\bar{s}_1,\ldots,\bar{s}_{2r}$}\\
&= \mbox{the $H^*_G(\pt;\Q)$-submodule
of $H_G^*\bigl((\PP^1)^{2r};\Q\bigr)$ generated
by $\bar{s}_0,\ldots,\bar{s}_{2r}$}.\\
\end{split} 
\end{equation*}
\end{enumerate}
\end{theorem}

\begin{proof}
The fact that $\Phi_{2r}^*: H^*_G(F_{2r};\Q) \to H^*_G((\PP^1)^{2r};\Q)$ is
  injective follows from the commutative diagram 
\begin{diagram}
H^*_G(F_{2r};\Q) &\rTo^{\Phi_{2r}^*}  & H^*_G((\PP^1)^{2r};\Q) \\ 
\dTo&&\dTo \\ 
H^*_T(F_{2r};\Q) & \rTo^{\Phi_{2r}^*} & H^*_T((\PP^1)^{2r};\Q) \\
\end{diagram}
and the fact that both vertical arrows and the bottom horizontal arrow
are injective. 
For part (2), recall from Theorem~\ref{HTF2rQ} that 
$$H_T^*(F_{2r};\Q)\cong
\mbox{the $H^*_T(\pt;\Q)$-submodule
of $H_T^*\bigl((\PP^1)^{2r};\Q\bigr)$ generated
by $\bar{s}_1,\ldots,\bar{s}_{2r}$}.$$
It is clear that an $H_T^*(\pt;\Q)$-linear combination of 
$\bar{s}_0,\ldots,\bar{s}_{2r}$
is Weyl invariant if and only if its coefficients lie in
$\bigl(H_T^*(\pt);\Q\bigr)^W=H^*_G(\pt;\Q)$.
Thus we have
\begin{align*}
\im\Phi_{2r}\subset \bigl(H_T^*(F_{2r};\Q)\bigr)^W
&\cong\bigl(\mbox{$H_T^*(\pt;\Q)$-submodule generated
by $ \bar{s}_1,\ldots,\bar{s}_{2r}$}\bigr)^W\cr
&\cong \mbox{$H_G^*(\pt;\Q)$-submodule generated
by $ \bar{s}_0,\ldots,\bar{s}_{2r}$}.\cr
\end{align*}
The result follows. 
\end{proof}

\section{Computation of $\Phi_{2r}^*: K^*_G(F_{2r}) \to  K^*_G((\PP^1)^{2r})$}
\label{section:ImPhiK}

We now come to the technical heart of this manuscript, which is the
explicit computation of $K^*_T(F_{2r})$ and $K^*_G(F_{2r})$. By
results in \cite{HJS12} we know that $K^*_G(\Omega G)$ is the inverse
limit
of~$K^*_G(F_{2r})$ as $r \to \infty$, so knowledge of
$K^*_G(F_{2r})$ is a key step in the computation of $K^*_G(\Omega G)$.
This section is long, so we have divided the exposition into pieces.

\subsection{Preliminaries and general setup}\label{subsec:general} 

We first prove that, similar to the case of cohomology in the previous
section, the map $\Phi_{2r}:
(\PP^1)^{2r}\to(F_2)^r\to F_{2r}$ 
 induces an injection in equivariant $K$-theory.

\begin{proposition}\label{proposition:injective in KG}
Let $\Phi_{2r}:
(\PP^1)^{2r}\to F_{2r}$ be the map given in
Definition~\ref{definition:Phi 2r}. 
Then 
\begin{equation}\label{eq:injective-in-KG}
\Phi^*_{2r}: K^*_G(F_{2r})  \to K^*_G\bigl((\PP^1)^{2r}\bigr)
\end{equation}
is an injective ring homomorphism, and similarly
\begin{equation}\label{eq:injective-in-KT}
\Phi^*_{2r}: K^*_T(F_{2r}) \to K^*_T\bigl((\PP^1)^{2r}\bigr)
\end{equation}
is an injective ring homomorphism. 
\end{proposition}

\begin{proof}

Consider the commutative diagram
\begin{diagram}
K^*_G (F_{2r}) &\rTo^{\Phi_{2r}^*}  & K^*_G \Bigl ((\PP^1)^{2r} \Bigr ) 
= K^*_G(\pt)[L_1, \dots, L_{2r}]/\mathord{\sim} \cr
\dTo_{\chG}&&\dTo_{\chG}\cr
H^{\hp}_G(F_{2r}; \Q)  & \rTo^{\Phi_{2r}^*} &
H^{\hp}_G \Bigl ( (\PP^1)^{2r}; \Q \Bigr )
= H^{\hp}(\pt;\Q) [\bar{L}_1, \dots, \bar{L}_{2r}]/\mathord{\sim} \cr
\end{diagram}
By Theorem~\ref{HFG2rQ} we know that
the bottom horizontal map is an injection.
The vertical maps are Chern character maps
and so are injective since all the groups are torsion free.
This implies that 
the upper horizontal map must also be injective.
The proof for $K^*_T$ is identical. 
\end{proof}

By the above proposition, in order to compute $K^*_T(F_{2r})$ and $K^*_G(F_{2r})$, it therefore
remains to compute their images in $K^*_T((\PP^1)^{2r})$ and
$K^*_G((\PP^1)^{2r})$ respectively. We will accomplish this by an
induction argument on the variable $r$. As a first step, we note the
following.

\begin{lemma}\label{lemma:induction on r geometric} 
The diagram
\begin{diagram}\label{eq:Tdiagram}
(\PP^1)^{2r-2}&\rTo^{\Phi_{2(r-1)}}&F_{2(r-1)} \cr
\dTo &&\dTo  \cr
(\PP^1)^{2r}&\rTo^{\Phi_{2r}}& F_{2r} \cr
\end{diagram}
commutes, where the right vertical arrow is the 
canonical inclusion and the left 
vertical arrow is the inclusion $(x_1, x_2, \ldots, x_{2r-2}) \mapsto
(x_1, \ldots, x_{2r-2}, [T], [T]) \in (\PP^1)^{2r} \cong (G/T)^{2r}$. 
\end{lemma}

\begin{proof}
  This follows from the definition of the maps $\Phi_{2r}$ and
  Remark~\ref{remark:F2}. 
\end{proof}

\begin{remark}\label{remark:not G diagram} 
The left vertical map is not a $G$-equivariant map since $([T],[T])$
is not a $G$-fixed point, as in Remark~\ref{rk1}. All other maps
in~\eqref{eq:Tdiagram} are $G$-equivariant. 
\end{remark}

Let $\kappa: F_{2r} \to \Thom(\tau^{2r-1})$ be the
composition of the projection $F_{2r} \to F_{2r}/F_{2r-2}$ with the
$G$-equivariant homeomorphism $F_{2r}/F_{2r-2} \cong
\Thom(\tau^{2r-1}$ discussed in Proposition~\ref{prop:filtration
  quotient as Thom space}. By Lemma~\ref{lemma:induction on r geometric} we may consider the following 
commutative diagram, which provides the framework for all the arguments
in this section. 
\begin{equation} \label{d22}
\begin{diagram} 
0&\rTo&\tilde{K}^*_T\bigl(\Thom(\tau^{2r-1})\bigr)&\rTo^{\kappa^*}&K^*_T(F_{2r}) &
\rTo^{i^*}& K^*_T(F_{2r-2})&\rTo&0\cr
&&&&\dTo^{\Phi_{2r}^*} && \dTo_{\Phi_{2r-2}^*}&&\cr
&&&&K^*_T\bigl((\PP^1)^{2r}\bigr) &\rTo^{i^*}&
K^*_T\bigl((\PP^1)^{2r-2}\bigr)&&\cr
\end{diagram}
\end{equation}
Here it is important that we use $T$-equivariant 
$K$-theory instead of $G$-equivariant $K$-theory, since by
Remark~\ref{remark:not G diagram} the diagram~\eqref{eq:Tdiagram} is
not a diagram of $G$-equivariant maps.

\begin{remark}
  In our computation of $K^*_G(F_{2r})$ below, we occasionally use
  expressions such as $\ch_G(y)$, when the element $y \in
  K^*_T(F_{2r})$ happens to be Weyl-invariant. This is justified by
  the fact that our results in
  \cite{HJS12} show that $K^*_G(F_{2r})$ is isomorphic to
  $K^*_T(F_{2r})^W$; in this situation there is no 
  harm in using $\ch_G(y)$, since $\ch_G(y)$ determines $\ch_T(y)$. 
\end{remark}

In order to describe the image of $\Phi_{2r}^*: K^*_T(F_{2r}) \to K^*_T((\PP^1)^{2r})$ (respectively for $K^*_G$), we 
note first that for any $r>0$ there is a
natural $S_{2r}$-action on $(\PP^1)^{2r}$, commuting with the given
(diagonal) $G$-action on $(\PP^1)^{2r}$, obtained by interchanging the
factors. This geometric action induces an action on
$K^*_T((\PP^1)^{2r})$ (and $K^*_G((\PP^1)^{2r})$). 

\begin{definition}\label{definition:symmetric subring} 
Fix   $r \in \ZZ^+$. Let
$\Bigl(K^*_T\bigl( (\PP^1)^{2r} \bigr) \Bigr)^{S_{2r}}$ 
denote the subring which is invariant
under the $S_{2r}$-action above.  We call this the 
\textbf{symmetric subring of $K^*_T((\PP^1)^{2r})$}. 
We will use similar notation and terminology 
for statements with $K^*_G$ replacing $K^*_T$.
\end{definition} 

Motivated by Definition~\ref{definition:symmetric subring} and following our notation~\eqref{eq:def barsj} in cohomology, 
let $s_j$ 
 denote the element 
\begin{equation}\label{eq:def sj}
s_j(L_1, \ldots, L_{2r}) \in K^*_G((\PP^1)^{2r}) \subseteq
K^*_T((\PP^1)^{2r}),
\end{equation}
where $s_j(L_1, \ldots, L_{2r})$ is the 
$j$-th elementary symmetric polynomial in  $L_1, \ldots,
L_{2r}$. From the  relations $L^2_j = v L_j - 1$ in Theorem~\ref{theorem:P1summary},
 it follows that any element in
$\Bigl(K^*_T\bigl((\PP^1)^{2r}\bigr)\Bigr)^{S_{2r}}$ can be written as 
an $R(T)$-linear
combination of $\{s_0, \ldots, s_{2r}\}$. Thus we 
can also refer to $\Bigl(K^*_T\bigl((\PP^1)^{2r}\bigr)\Bigr)^{S_{2r}}$ as
the ring of \textbf{symmetric
polynomials} in $L_1, \ldots, L_{2r}$.

We can now state the main result of this section, which is that 
the image $\Phi_{2r}^*\bigl(K^*_T(F_{2r})\bigr)$ is
precisely equal to the symmetric subring
$\Bigl(K^*_T\bigl( (\PP^1)^{2r} \bigr) \Bigr)^{S_{2r}}$, in analogy
with the result in cohomology in Section~\ref{sec:computation cohomology}.

\begin{theorem} \label{KTF2r}
Let $G=SU(2)$ and $T \subset SU(2)$ its maximal torus. Let $\Phi_{2r}:
(\PP^1)^{2r} \to F_{2r}$ be the map given in
Definition~\ref{definition:Phi 2r}. Then 
\begin{equation*}
\begin{split} 
K^*_T(F_{2r}) & \cong\Phi_{2r}^*\bigl(K_T^*(F_{2r})\bigr) \\
& = \Big(K^*_T((\PP^1)^{2r})\Big)^{S_{2r}} \\ 
&=\{\mbox{symmetric polynomials in $L_1, \ldots, L_{2r}$
  in~$K^*_T\bigl((\PP^1)^{2r}\bigr)$}\} \\
&= \mbox{the $K^*_T(\pt)$-subalgebra
of $K^*_T\bigl((\PP^1)^{2r}\bigr)$ generated
by $s_1,\ldots,s_{2r}$}\\
&= \mbox{the $K^*_T(\pt)$-submodule
of $K^*_T\bigl((\PP^1)^{2r}\bigr)$ generated
by $s_0,\ldots,s_{2r}$}.\\
\end{split} 
\end{equation*}
\end{theorem}

As a first step towards the proof of Theorem~\ref{KTF2r}, we prove
containment in one direction.

\begin{lemma}\label{lemma:inclusion}
The image $\Phi_{2r}^*(K^*_T(F_{2r}))$ is contained in the symmetric
subring $\bigl(K^*_T((\PP^1)^{2r})\bigr)^{S_{2r}}$ of~${K}^*_T\bigl((\PP^1)^{2r}\bigr)$. 
\end{lemma}

\begin{proof} 
Let $x \in \Phi^*_{2r}(K^*_T(F_{2r})) \subset K^*_T((\PP^1)^{2r})$. To show $\sigma(x)=x$ for all $\sigma \in S_{2r}$, 
it suffices to consider the case where $\sigma$ is a transposition.
 The map $\sigma$ is induced by a $G$-map
of spaces which permutes factors within $(\PP^1)^{2r}$.
Thus the question reduces to showing that the diagram 
\begin{diagram} 
 (\PP^1)^{2r} & \rTo^{\sigma} &      (\PP^1)^{2r} \\
\dTo^{{\Phi_{2r}}}  &  &  \dTo_{{\Phi_{2r}}} \\
F_{2r} & \rEqualto & F_{2r} \\
\dTo &  & \dTo \\
\Omega G  & \rEqualto & \Omega G
\end{diagram}
$G$-homotopy commutes.
The standard proof that $\Omega G$ is homotopy-abelian goes
through $G$-equivariantly to prove this.
\end{proof} 

The remainder of this section is devoted to proving the reverse
inclusion
\begin{equation}\label{eq:reverse inclusion}
\Bigl(K^*_T\bigl( (\PP^1)^{2r} \bigr) \Bigr)^{S_{2r}}
 \subseteq \Phi_{2r}^*(K^*_T(F_{2r}))
\end{equation}
by an inductive argument. 
The following is straightforward, giving us a sufficient condition for
proving~\eqref{eq:reverse inclusion} inductively. 

\begin{lemma}\label{lemma:sufficient}
Fix   $r \in \ZZ^+ $. Suppose that $K^*_T((\PP^1)^{2r-2})^{S_{2r-2}}$ is contained in 
$\Phi^*_{2r-2}(K^*_T(F_{2r-2}))$. Let $i^*$ be the bottom horizontal arrow in diagram~\eqref{d22}. 
If 
\begin{equation}\label{eq:ker istar cap symmetrics} 
\Ker i^* \cap 
\Bigl(K^*_T\bigl( (\PP^1)^{2r} \bigr) \Bigr)^{S_{2r}}
 \subseteq \Phi_{2r}^*(K^*_T(F_{2r}))
\end{equation}
then 
\[
\Bigl(K^*_T\bigl( (\PP^1)^{2r} \bigr) \Bigr)^{S_{2r}}
 \subseteq \Phi_{2r}^*(K^*_T(F_{2r})).
\]
\end{lemma} 

\begin{proof} 
The inclusion map $i: (\PP^1)^{2r-2} \into
(\PP^1)^{2r}$ is $S_{2r-2}$-equivariant, where $S_{2r-2} \subseteq
S_{2r}$ is the subgroup of $S_{2r}$ which only moves the first $2r-2$
factors. Thus the induced map $i^*$ on equivariant $K$-theory
appearing has the property that 
the image under $i^*$ of the symmetric subring of
$K^*_T((\PP^1)^{2r})$ is contained in
 $\Bigl(K^*_T\bigl( (\PP^1)^{2r-2} \bigr) \Bigr)^{S_{2r-2}}$.
The claim now follows from a straightforward diagram chase using 
the inductive hypothesis. 
\end{proof}

In fact, we will prove a stronger result than~\eqref{eq:ker istar cap
  symmetrics}, recorded in
Proposition~\ref{proposition:main technical} below, for the
statement of which we need some notation. 
In particular, 
it will be useful for the remainder of the discussion to choose specific
generators of $\tilde{K}_T(\Thom(\tau^{2r-1}))$ as follows. Consider
the diagram~\eqref{maindiagram} and the corresponding maps $\tilde{p},
\bar{p}, j$, and $k$ for the case $X = \PP^1$ and $\zzeta =
\tau^{2r-1}$. 
Let
$\gamma$ be the canonical line bundle over $\PP^1$.
Define 
\begin{equation} \label{defx}
x:=\bar{p}^*(\gamma)-1 \in K^*_G\bigl(\PP (\tau^{2r-1}\oplus\epsilon)\bigr)
\subset K^*_T\bigl(\PP (\tau^{2r-1}\oplus\epsilon)\bigr).\end{equation}
Let $U := U_{\tau^{2r-1}} = \lambda\bigl(\gamma_{\tau^{2r-1}\oplus\epsilon}^*\otimes\bar{p}^*(\tau^{2r-1})\bigr)$
denote the Thom class of $\tau^{2r-1}$ as in~\eqref{def eqvt Thom
  class}. 
The equivariant Thom isomorphism 
  says in this setting that 
$$\tilde{K}^*_G(F_{2r}/F_{2r-2})=\tilde{K}^*_G\bigl(\Thom(\tau^{2r-1})\bigr)$$
is freely generated as an $R(G)$-module by $U$ and~$xU$. We will use these
generators in our arguments below.

We can now state the main technical proposition of this section. 

\begin{proposition}\label{proposition:main technical}
  The map $\Phi_{2r}^* \circ \kappa^*: \tilde{K}^*_T\bigl(\Thom(\tau^{2r-1})\bigr)
  \to K_T\bigl((\PP^1)^{2r}\bigr)$ induces an isomorphism of
  $\tilde{K}^*_T(\Thom(\tau^{2r-1}))$ onto the subspace 
\[
\Ker i^* \cap 
\Bigl(K^*_T\bigl( (\PP^1)^{2r} \bigr) \Bigr)^{S_{2r}}
\subseteq
K^*_T\bigl((\PP^1)^{2r}\bigr).
\]
Specifically, 
the images of~$U$ and~$xU$ under~$\Phi_{2r}^* \circ \kappa^*$ form an $R(T)$-module basis for
$\Ker i^* \cap \Bigl(K^*_T\bigl( (\PP^1)^{2r} \bigr) \Bigr)^{S_{2r}}$.
In particular, 
$$\Ker i^*\cap K^*_T\bigl((\PP^1)^{2r}\bigr)^{S_{2r}}
 \subset\im \Phi_{2r}^*.$$
\end{proposition}

The proof of Proposition~\ref{proposition:main technical} is both long
and technical, and occupies Sections~\ref{subsec:construct bases}
to~\ref{subsection:finish proof}. Thus,
before embarking on the details, we briefly sketch the main ideas of
the proof. 
Our strategy is to relate the
map $\Phi_{2r}^*\circ \kappa^*$ to its analogue in $H^{\pi}_T$ via the homomorphism
$\chT$. This method allows us to take advantage of the presence of a
$\Z$-grading in cohomology. 
The
following
commutative diagram 
\begin{equation}\label{eq:main diagram basis-free}
\xymatrix @C=2.5pc {
\tilde{K}^*_T\bigl(\Thom(\tau^{2r-1})\bigr) \ar[rrr]^{\Phi_{2r}^* \circ \kappa^*}
\ar[d]^{\chT}  & & &
\Ker i^* \cap \Bigl({\tilde{K}}^*_T\bigl((\PP^1)^{2r}\bigr)\Bigr)^{S_{2r}} 
\ar[d]^{\chT} \\
\tilde{H}^{\pi}_T(\Thom(\tau^{2r-1});\QQ)\ar[rrr]^{{\Phi_{2r}}^* \circ \kappa^*}
&&& \Ker i^* \cap\Bigl({\tilde{H}}^{\pi}_T\bigl((\PP^1)^{2r};\QQ
\bigr)\Bigr)^{S_{2r}}. \\
}\end{equation}
will be central in our analysis. Note that the diagram is well-defined
since, by naturality, $\chT$ takes $S_{2r}$-invariant elements to
$S_{2r}$-invariant elements, and also takes $\Ker i^*$ to $\Ker i^*$. 
Moreover, the cohomology version of~\eqref{d22} shows that the image
of $\Phi_{2r}^* \kappa^*$ on
$\tilde{H}^{\pi}_T(\Thom(\tau^{2r-1});\Q)$ lie in $\Ker i^*$, and
Theorem~\ref{HTF2rQ} shows that they are symmetric. 
Also note that the top horizontal arrow in~\eqref{eq:main diagram basis-free} is a
morphism of $R(T)$-modules, while the bottom horizontal arrow is a
morphism of $H^{\pi}_T(\pt)$-modules. The vertical arrows
satisfy the relation 
$$ \chT(\rho m)= \chT(\rho) \chT(m)$$
for all $\rho \in R(T)$ and any $m$ in the domain. 
Our goal, stated in terms of~\eqref{eq:main diagram basis-free}, is to prove that the top horizontal arrow is an
isomorphism.

We will accomplish this goal by concrete linear algebra. Recall that
$\tilde{K}^*_T\bigl(\Thom(\tau^{2r-1})\bigr)$ is a free $R(T)$-module of rank
$2$, where we have fixed a choice of basis $\{U, xU\}$. Letting
$\bar{U}$ denote the (equivariant) cohomology Thom class of
$\tau^{2r-1}$ and 
\begin{equation} \label{defbarx}\bar{x}:=c_1^G(x) = c_1^G\bigl(\bar{p}^*(\gamma)\bigr)
\in H^2_G\bigl(\PP (\tau^{2r-1}\oplus\epsilon)\bigr)
\subset H^2_T\bigl(\PP (\tau^{2r-1}\oplus\epsilon)\bigr)
\end{equation}
it is also clear from the 
(equivariant) cohomology Thom isomorphism that $\{\bar{U},
\bar{x}\bar{U}\}$ form a basis for
$\tilde{H}_T^{\pi}(\Thom(\tau^{2r-1});\Q)$ as a free
$\tilde{H}_T^{\pi}(\pt;\Q)$-module. We show in
Section~\ref{subsec:construct bases} that both $\Ker i^* \cap
(K^*_T\bigl((\PP^1)^{2r})\bigr)^{S_{2r}}$ and $\Ker i^* \cap
\bigl(H^{\pi}_T((\PP^1)^{2r};\Q)\bigr)^{S_{2r}}$ are also free rank-$2$ modules over
$R(T)$ and $H^{\pi}_T(\pt;\Q)$ respectively, and find explicit bases
$\{K_1, K_2\}$ and $\{\bar{K}_1, \bar{K}_2\}$ respectively.

Given these choices of bases, we can construct a $ 2 \times
2$ matrix determining any of the four maps in the
diagram~\eqref{eq:main diagram basis-free}. For example, for the right
vertical arrow, we may write 
\begin{equation}\label{eq:def N}
\chT(K_1) = n_{11} \bar{K}_1 + n_{21} \bar{K}_2, 
\quad
\chT(K_2) = n_{12} \bar{K}_1 + n_{22} \bar{K}_2
\end{equation}
where $n_{ij} \in H^{\pi}_T(\pt;\QQ)$. This defines a $2 \times 2$
matrix $\mathbf{N} = (n_{ij})$. Similarly we may define matrices
$\mathbf{M}$ corresponding to the left vertical arrow, $Q$ for the top
horizontal arrow, and $\bar{Q}$ for the bottom horizontal arrow. Note
that the entries of $\mathbf{N}, \mathbf{M}, \bar{Q}$ are all in
$H^{\pi}_T(\pt;\QQ)$, whereas the entries of $Q$ are in $R(T)$. 
We record these definitions schematically in the following diagram. 
\begin{equation}\label{eq:main diagram with basis} 
\xymatrix{
\langle U, xU \rangle \ar[r]^Q \ar[d]^{\bf M} & \langle K_1, K_2 \rangle \ar[d]^{{\bf N} } \\
\langle \bar{U}, \bar{x}\bar{U} \rangle \ar[r]^{\bar{Q}} & \langle \bar{K}_1, \bar{K}_2 \rangle 
}
\end{equation}
The commutativity of~\eqref{eq:main diagram basis-free} implies that
these matrices satisfy
\begin{equation}\label{matrixeqn}
{\bf N}\, \ch_T(Q) = \bar{Q}\, {\bf M}
\end{equation}
where the notation $\chT(Q)$ denotes the $2 \times 2$ matrix obtained
by applying $\chT$ to each entry of~$Q$.

With this notation in place it is immediate that the following is sufficient to
prove Proposition~\ref{proposition:main technical}: 
\begin{equation}\label{eq:Q has det 1}
\begin{minipage}{0.8\linewidth}
the matrix $Q$ in~\eqref{eq:main diagram with basis} has determinant
$1$ (and is hence invertible). 
\end{minipage}
\end{equation}
We will prove the claim in~\eqref{eq:Q has det 1} via the indirect
route of computing $\bf{M}, \bf{N}$, and $\bar{Q}$, and then using the
relation~\eqref{matrixeqn} to deduce that the determinant of $\chT(Q)$
is $1$. This implies $\det(Q)=1$ since $\chT$ is injective. More
specifically, in Section~\ref{subsection:M and N} we explicitly
compute both $\bf{M}$ and $\bf{N}$.
We compute $\bar{Q}$ and some determinants to finish the proof of
Proposition~\ref{proposition:main technical} (and hence
Theorem~\ref{KTF2r}) in Section~\ref{subsection:finish proof}.

\subsection{Module bases for $\Ker i^* \cap
(K^*_T((\PP^1)^{2r}))^{S_{2r}}$ and $\Ker i^* \cap
(H^{\pi}_T((\PP^1)^{2r};\Q))^{S_{2r}}$}\label{subsec:construct bases}

As  above, we must prove that both 
$\Ker i^* \cap
(K^*_T((\PP^1)^{2r}))^{S_{2r}}$ and $\Ker i^* \cap
(H^{\pi}_T((\PP^1)^{2r};\Q))^{S_{2r}}$ are free rank-$2$ modules over
appropriate rings, and then to fix particular choices of module bases
for each. (In fact, we will not specify the bases completely, since
for our later arguments only the `highest-order terms' are needed.)

We
begin with an explicit description of the map $i^*$ in terms of the
presentations of $K^*_T((\PP^1)^{2r})$ and $K^*_T((\PP^1)^{2r-2})$
given in Theorem~\ref{theorem:P1summary}.
Let $\{L_1, \ldots, L_{2r}\}$ denote the generators of
$K^*_T((\PP^1)^{2r})$ as before, and let $\{L_1',\ldots, L_{2r-2}' \}$
denote the generators
of~$K^*_T\bigl((\PP^1)^{2r-2}\bigr)$. With respect to these variables 
the map 
$i^*: K^*_T\bigl((\PP^1)^{2r}\bigr)\to K^*_T\bigl((\PP^1)^{2r-2}\bigr)$
is given by
\begin{equation}\label{eq:iota star of L_j}
i^* (L_j) = \begin{cases} L'_j & \mbox{ if $j \le 2r-2$};\cr
b^{-1} & \mbox{if $j = 2r-1$}; \cr
b & \mbox{if $j = 2r$.} \cr
\end{cases} 
\end{equation}
Also note that the quadratic relations $L_j^2 = vL_j-1$
in Theorem~\ref{theorem:P1summary} imply that any symmetric polynomial
in the $L_j$ 
in $K^*_T\bigl((\PP^1)^{2r}\bigr)$ can be
expressed as an $R(T)$-linear combination of $s_0, \ldots, s_{2r}$.
Since we are interested in the kernel of $i^*$ restricted to the
symmetric polynomials, it is useful to compute $i^*$ on the $s_k$. 
Let $s'_0, \ldots, s'_{2r-2}$ denote the analogous elements in
$K^*_T((\PP^1)^{2r-2})$. Using the expression 
\begin{align*}
s_j(L_1,\ldots, L_{2r})
&=s_{j-2}(L_1,\ldots,L_{2r-2})L_{2r-1}L_{2r}
+s_{j-1}(L_1,\ldots,L_{2r-2})L_{2r-1}\cr
&\phantom{=}+s_{j-1}(L_1,\ldots,L_{2r-2})L_{2r}
+s_{j}(L_1,\ldots,L_{2r-2})
\end{align*}
it follows from a straightforward computation that 
\begin{equation}\label{eq:iota star of s_j}
i^* (s_j) = \begin{cases}
s'_0 &\mbox{if $j=0$};\cr
s'_1 +v s'_{0} &\mbox{if $j=1$};\cr
 s'_j +v s'_{j-1} + s'_{j-2}
& \mbox{if $1 < j \le 2r-2$};\cr
  v s'_{2r-2} + s'_{2r-3}
& \mbox{if $ j = 2r-1$};\cr       s'_{2r-2}
& \mbox{if $    j =2r$.}\cr
\end{cases}
\end{equation}

 The corresponding matrix is 
 \begin{equation} \label{k:matrix} 
 \begin{pmatrix}
 1 & v & 1 & \ldots & 0 & 0 & 0 & 0 & 0\cr
 0 & 1 & v & \ldots & 0 & 0 & 0 & 0 & 0\cr
 0 & 0 & 1 & \ldots & 0 & 0 & 0 & 0 & 0\cr
 \vdots & \vdots & \vdots & \ddots & \vdots & \vdots & \vdots & \vdots & \vdots
 \cr
 0 & 0 & 0 & \ldots & 1 & v & 1 & 0 & 0\cr
 0 & 0 & 0 & \ldots & 0 & 1 & v & 1 & 0\cr
 0 & 0 & 0 & \ldots & 0 & 0 & 1 & v & 1\cr
% 0 & 0 & 0 & \ldots & 0 & 0 & 0 & 0 & 0\cr
% 0 & 0 & 0 & \ldots & 0 & 0 & 0 & 0 & 0\cr
 \end{pmatrix}
 \end{equation}

From~\eqref{eq:iota star of s_j} and~\eqref{k:matrix} it follows that $\Ker(i^*) \cap \left( K^*_T((\PP^1)^{2r})\right)^{S_{2r}}$ is 
a free rank-$2$-module and that there exists a basis
$K_1, K_2$ of the form 
\begin{align*}  \label{kernel1} 
K_1 &= s_{2r-1} + \mbox{lower-order terms in $s_0, \ldots, s_{2r-2}$} \cr
\noalign {and} \cr
K_2 &= -s_{2r}+ \mbox{lower-order terms in $s_0, \ldots, s_{2r-2}$}. \cr
\end{align*}
(Only the leading terms $s_{2r-1}$ and $s_{2r}$ of $K_1$ and $K_2$ are
of concern to us, so we do not record details about the lower-order
terms.)

We now make analogous computations for 
$i^*: H^{\pi}_T\bigl((\PP^1)^{2r};\Q\bigr)\to
H^{\pi}_T\bigl((\PP^1)^{2r-2};\Q\bigr)$.
Noting that $v=b+b^{-1}$ and that $\bar{b}(-\bar{b})=-\bar{t}$,
we have the following equations for $i^*(\bar{L}_j)$ analogous to~\eqref{eq:iota star of L_j}
with $b$ and $b^{-1}$ replaced by $\bar{b}$ and~$-\bar{b}$ respectively.
\begin{equation}\label{eq:iota star of bars_j}
i^* (\bar{s}_j) = \begin{cases}
\bar{s}'_0 &\mbox{if $j=0$};\cr
\bar{s}'_1 &\mbox{if $j=1$};\cr
\bar{s}'_j -\bar{t}  \bar{s}'_{j-2}
& \mbox{if $1 < j \le 2r-2$};\cr
   -\bar{t} \bar{s}'_{2r-3} & \mbox{if $ j = 2r-1$};\cr       
- \bar{t}\bar{s}'_{2r-2} & \mbox{if $    j =2r$.}\cr
\end{cases}
\end{equation}
The corresponding matrix for $i^*$ is
\begin{equation} \label{h:matrix} \begin{pmatrix}
1 & 0 & -\bar{t} & \ldots & 0 & 0 & 0 & 0 & 0\cr
0 & 1 & 0 & \ldots & 0 & 0 & 0 & 0 & 0\cr
0 & 0 & 1 & \ldots & 0 & 0 & 0 & 0 & 0\cr
\vdots & \vdots & \vdots & \ddots & \vdots & \vdots & \vdots & \vdots & \vdots
\cr
0 & 0 & 0 & \ldots & 1 & 0 & -\bar{t} & 0 & 0\cr
0 & 0 & 0 & \ldots & 0 & 1 & 0 & -\bar{t} & 0\cr
0 & 0 & 0 & \ldots & 0 & 0 & 1 & 0 & -\bar{t}\cr
%0 & 0 & 0 & \ldots & 0 & 0 & 0 & 0 & 0\cr
%0 & 0 & 0 & \ldots & 0 & 0 & 0 & 0 & 0\cr
\end{pmatrix}
\end{equation}
Again it follows that 
$\Ker(i^*) \cap \left( H^{\pi}_T((\PP^1)^{2r}; \Q)\right)^{S_{2r}}$ is a free rank-$2$ module and 
that there exists a basis of the form 
\begin{align*}
\bar{K}_1 &= \bar{s}_{2r-1}  + \mbox{lower-order terms in $\bar{s}_0,\ldots, 
\bar{s}_{2r-2}$}\cr
 \noalign {and} \cr
\bar{K}_2 &=-\bar{s}_{2r}+ \mbox{lower-order terms in $\bar{s}_0,\ldots, \bar{s}_{2r-1}$}                           \cr
\end{align*}
where both $\bar{K}_1$ and $\bar{K}_2$ are homogeneous.

\subsection{Computation of the matrices $\bf{N}$ and
  $\bf{M}$}\label{subsection:M and N}

We next turn to a computation of the matrix $\bf{N}$ in~\eqref{eq:main
diagram with basis}, the entries of which are determined
by~\eqref{eq:def N}. Hence our task is to compute $\chT(K_1)$ and
$\chT(K_2)$ in terms of $\{\bar{K}_1, \bar{K}_2\}$, with respect to
the bases chosen in Section~\ref{subsec:construct bases}. Since the
$K_i$ are written in terms of the $s_k$, we first compute
$\chT(s_k)$.

Note that in general if two variables $t$ and $y$ satisfy the relation
$y^2=t$, then the formal series $e^y = \sum_k \frac{y^k}{k!}$ can be expressed as 
\begin{equation}\label{eq:ey with p and q}
e^{y}=yp(t)+q(t)
\end{equation}
where 
\begin{equation} \label{defpt}
p(t):= \sum_{k=0}^\infty \frac{t^k}{(2k+1)!} =
\sinh(\sqrt{t})/\sqrt{t} 
\end{equation}
and 
\begin{equation} \label{defqt}
q(t):= \sum_{k=0}^\infty \frac{t^k}{(2k)!} = \cosh(\sqrt{t}).
\end{equation}
In our setting, recall that $\bar{L}_k$ satisfies
$\bar{L}_k^2=\bar{t}$. Using~\eqref{eq:ey with p and q} we obtain, by
definition of the Chern character, 
$$\chT(L_k)=e^{\bar{L}_k}=
\bar{L}_k p(\bar{t})+q(\bar{t})$$
so $\chT(L_k)$ is an expression in $H^{\pi}((\PP^1)^{2r};\Q)$ which is
linear in $\bar{L}_k$ with coefficients in $H^{\pi}_T(\pt;\Q)$. 
Therefore we conclude 
\begin{align*}
\chT(s_k)&=s_k\bigl(\chT(L_1),\ldots,\chT(L_{2r-2})\bigr)\cr
&=s_k\bigl(\bar{L}_1p(\bar{t})+q(\bar{t}),\ldots,
\bar{L}_{2r}p(\bar{t})+q(\bar{t})\bigr)\cr
&=p(\bar{t})^k\bar{s}_k+\mbox{lower-order terms in $\bar{s}_0,\ldots, \bar{s}_{k-1}$}.\cr
\end{align*}
It then follows that
\begin{align*}
\chT(K_1)&=p(\bar{t})^{2r-1}\bar{s}_{2r-1}
+\mbox{lower-order terms in $\bar{s}_0,\ldots, \bar{s}_{2r-2}$}\cr
\noalign{\rm and}
\chT(K_2)&=-p(\bar{t})^{2r}\bar{s}_{2r}
+\mbox{lower-order terms in $\bar{s}_0,\ldots, \bar{s}_{2r-1}$}.
\end{align*}
Comparing the coefficients of $\bar{s}_{2r-1}$ and~$\bar{s}_{2r}$
in $\chT(K_1)$ and $\chT(K_2)$ with those in $\bar{K}_1$
and~$\bar{K}_2$, 
we conclude that 
\[
\chT(K_1)=p(\bar{t})^{2r-1}\bar{K_1}
\]
and
\[
\chT(K_2)=p(\bar{t})^{2r}\bar{K_2} + A(\bar{t})\bar{K}_1
\]
for some $A(\bar{t})\in H_T^{\pi}(\pt;\Q)$.
Hence we have 
\begin{equation} \label{ndef}  { \bf N} = \begin{pmatrix} 
\bigl ( p(\bar{t})\bigr ) ^{2r-1} &   A(\bar{t}) \cr 
0 & \bigl ( p(\bar{t}) \bigr ) ^{2r } \cr \end{pmatrix}.
\end{equation}

We now turn to a computation of $\bf{M}$, for which we must first
compute the Chern character of $U$ and $xU$. We will in fact compute
$\chG(U)$ and $\chG(xU)$, from which we can deduce $\chT(U)$ and
$\chT(xU)$. For this computation it is useful to introduce the
bundle 
\[
a:=\gamma^*_{\tau^{2r-1}\oplus\epsilon}\otimes\bar{p}^*(\tau)
\]
on $\Thom(\tau^{2r-1})$. Then it follows from~\eqref{def eqvt Thom
  class}
that $U=\lambda(a^{2r-1})$. Letting $\bar{a}$ denote the first Chern
class $c_1(a)$, it also follows from the definition of
$a$, Corollary~\ref{cor:Cherntau},
and the definition~\eqref{defbarx} of $\bar{x}$ 
that 
\begin{equation}\label{eq:c1 relation a and barx}
c_1^G(\gamma_{\tau^{2r-1}\oplus \epsilon}) = - \bar{a} - 2 \bar{x}.
\end{equation}

We are now in a position to compute $\chG(U)$ in terms of $\bar{a}$. 
Recall that the definition of the Chern character implies that 
if $\xi=\xi_1+\ldots+ \xi_n$ is a sum of equivariant line bundles then
$$\ch_G\bigl(\lambda(\xi)\bigr)=
\sum_{k=0}^n (-1)^{k}s_k(e^{\overline{\xi_1}},\ldots,e^{\overline{\xi_n}})$$
where $\overline{\xi_j} := c_1^G(\xi_j). $
Applying this to $U = a^{2r-1}$ yields 
\begin{align*}
\chG(U)&=\sum_{k=0}^{2r-1} (-1)^ks_k(e^{\bar{a}},e^{\bar{a}},\ldots,e^{\bar{a}})\cr
&=1-(2r-1)e^{\bar{a}}+{2r-1\choose 2}e^{2\bar{a}}+\ldots=(1-e^{\bar{a}})^{2r-1}\cr
&=(1-e^{\bar{a}})^{2r-1}.\cr
\end{align*}

In order to compute the matrix ${\bf M}$ with respect to the chosen
bases $\{U, xU\}$ and $\{\bar{U}, \bar{x}\bar{U}\}$, we must now
relate $\bar{a}$ to $\bar{U}$ and $\bar{x}\bar{U}$. The next two
lemmas serve this purpose.

\begin{lemma}\label{lemma:relation on bar-a}
$\bar{a}^{2r}+2\bar{x} \bar{a}^{2r-1} = 0.$
\end{lemma} 

\begin{proof} 
For an equivariant bundle $\xi$, we write
$c^G(\xi)=c_0^G(\xi)+c_1^G(\xi)+\ldots$ for its total (equivariant)
Chern class. Then the Whitney sum formula and the defining
equation~\eqref{basemainequation} for the bundle
$\beta_{\tau^{2r-1}\oplus\epsilon}$ together imply 
$$c^G(\bar{p}^*\tau^{2r-1})=
c^G(\gamma_{\tau^{2r-1}\oplus\epsilon})c^G(\beta).$$
Since $c^G(\bar{p}^*(\tau^{2r-1}) = c^G(\bar{p}^*\tau)^{2r-1} =
(1-2\bar{x})^{2r-1}$ and $c^G(\gamma_{\tau^{2r-1}\oplus \epsilon}) =
1-(\bar{a}+2\bar{x})$ by~\eqref{eq:c1 relation a and barx}, we
conclude 
\begin{equation} \label{e:e1}
c^G(\beta)=c^G(\bar{p}^*(\tau^{2r-1}))\bigl(c^G(\gamma_{\tau^{2r-1}\oplus\epsilon}
)\bigr)^{-1}
= \frac{(1-2\bar{x})^{2r-1}}{1-(\bar{a}+2\bar{x})} 
=(1-2\bar{x})^{2r-1} \sum_{m=0}^\infty (\bar{a}+2\bar{x})^m. 
\end{equation}
Taking the degree-$k$ part of~\eqref{e:e1} yields
$$ c_k^G(\beta) =
\sum_{j=0}^{k} {2r-1\choose j} (-2\bar{x})^j (\bar{a}+2\bar{x})^{k-j} $$
and so 
$$c_k^G(\beta) = (\bar{a}+2\bar{x}) c_{k-1}^G(\beta)+{2r-1\choose k}(-2\bar{x})^k.$$
In particular, since $\dim(\beta)=2r-1$ and ${2r-1\choose 2r} = 0$, we
conclude 
\begin{equation} \label{twormonechern}
 c_{2r}^G (\beta) = (\bar{a}+2\bar{x})c_{2r-1}^G(\beta) . 
\end{equation}
Also, for $k = 2r-1$, 
\begin{equation}\label{eq:c 2r-1 of beta}
c_{2r-1}^G (\beta) = 
 \sum_{j=0}^{2r-1} {2r-1\choose j} (-2x)^j (\bar{a}+2\bar{x})^{2r-1-j}
= (\bar{a}+2\bar{x}-2\bar{x})^{2r-1} = \bar{a}^{2r-1}.
\end{equation}
Putting~\eqref{twormonechern}
and~\eqref{eq:c 2r-1 of beta} together yields the desired result. 
\end{proof}

Next recall 
that the cohomology Thom class and $K$-theory class are related by a
formula involving the Chern character and the Todd
class. 
Specifically, 
for an equivariant $n$-bundle~$\xi$, we have 
\begin{equation} \label{toddeqn}
\overline{U_\xi} = (-1)^n\bar{p}^*\bigl(\Todd_G(\xi)\bigr)\, \ch_G(U_\xi)\end{equation}
where 
$$\Todd_G(L):=\frac{c_1^G(L)}{1-e^{-c_1^G(L)}}$$
for an equivariant line bundle~$L$, and the definition is extended to higher-rank bundles by
means of the splitting principle \cite[pages 13--14]{Lan05}. 
We have the following. 

\begin{lemma}\label{lemma:bar-U in terms of bar-x bar-a}
$\bar{U} =-\bar{a}^{2r-1}$. 
\end{lemma} 

\begin{proof} 
Applying~\eqref{toddeqn} to the bundle 
$\tau^{2r-1}$ and using the multiplicativity of the Todd class yields
\begin{align*}
\bar{U} &= (-1)^{2r-1}\bar{p}^*\bigl(\Todd_G(\tau^{2r-1})\bigr)\, \chG(U) \cr
&= \Bigl(-\bar{p}^*\bigl(\Todd_G(\tau)\bigr) \Bigr)^{2r-1} \bigl(1-e^{\bar{a}}\bigr )^{2r-1}  \cr
&=\left( \frac{2\bar{x }    }{\bigl(1-e^{-2\bar{x}}\bigr )}\right)^{2r-1}
\bigl(1-e^{\bar{a}}\bigr )^{2r-1}\cr
&=\left( \frac{ (- \bar{a}) }{\bigl(1-e^{\bar{a}}\bigr )}\right)^{2r-1}
\bigl(1-e^{\bar{a}}\bigr )^{2r-1}\cr
&= (-1)^{2r-1} \bar{a}^{2r-1} \cr
&= - \bar{a}^{2r-1}. \cr
\end{align*}
where we have used Lemma~\ref{lemma:relation on bar-a} together with the fact
that $(1-e^{\bar{a}})^{2r-1}$ is divisible by~$\bar{a}^{2r-1}$.
\end{proof}

Using Lemmas~\ref{lemma:relation on bar-a} and~\ref{lemma:bar-U in
  terms of bar-x bar-a}, some algebraic manipulation (briefly sketched
below) allows us to
express $\chG(U) = (1-e^{\bar{a}})^{2r-1}$ in terms of $\bar{x}$ and
$\bar{U}$. Given a variable $y$, note that the expression 
$g(y):=\left(\frac{e^y-1}{y}\right)^{2r-1}$ can be rewritten as a sum
$g_1(y) + yg_2(y)$, where both $g_1$ and $g_2$ are even functions of
$y$, as follows: 
\begin{equation}
g(y) =e^{(2r-1)y/2}\left(\frac{\sinh(y/2)}{y/2}\right)^{2r-1}
=g_1(y^2) + y g_2(y^2)
\end{equation}
where
\begin{equation}\label{eq:def g1}
g_1(y)=\cosh\bigl((2r-1)\sqrt{y}/2\bigr)
\left(\frac{\sinh(\sqrt{y}/2)}{\sqrt{y}/2}\right)^{2r-1}
\end{equation}
and
\begin{equation}\label{eq:def g2}
g_2(y)=
\frac{\sinh\bigl((2r-1)\sqrt{y}/2\bigr)}{\sqrt{y}}
\left(\frac{\sinh(\sqrt{y}/2)}{\sqrt{y}/2}\right)^{2r-1}.
\end{equation}
Applying this to our situation, we get $\chG(U)=-\bar{a}^{2r-1}g(\bar{a})
=-\bar{a}^{2r-1}\bigl(g_1(\bar{a}^2)+\bar{a}g_2(\bar{a}^2)\bigr)$.
Also note that 
$\bar{x}^2=\bar{t}$, as can be seen from Lemmas~\ref{lemma:barcalL
  squared} and~\ref{relatingLtogamma} and the definition
of~$\bar{x}$. This, together with 
Lemmas~\ref{lemma:relation on bar-a} and~\ref{lemma:bar-U in terms of
  bar-x bar-a}, gives 
\begin{eqnarray} \label{reflabel}\chG(U)&=&-\bar{a}^{2r-1}\Bigl(
g_1\bigl((-2\bar{x})^2\bigr)+\bar{a}g_2\bigl((-2\bar{x})^2\bigr)
\Bigr) \cr
&=&-\bar{a}^{2r-1}\bigl(g_1(4\bar{t})-2\bar{x}g_2(4\bar{t})\bigr) = 
g_1(4\bar{t})\bar{U}-2g_2(4\bar{t})\bar{x}\bar{U}.
\end{eqnarray}

Since $\bar{x}^2=\bar{t}$, we also get 
\begin{align*}
\chG\bigl(xU\bigr)&=\bigl(\chG(x)\chG(U)\bigr)\cr
&=\bigl(p(\bar{t})\bar{x}+q(\bar{t})\bigl)
\bigl(g_1(4\bar{t}) \bigl (\bar{U}\bigr ) 
-2g_2(4\bar{t})\bigl(\bar{x}\bar{U}\bigr ) \bigr)\cr
&=q(\bar{t})g_1(4\bar{t})\bigl (\bar{U}\bigr ) -2\bar{t}p(\bar{t})g_2(4\bar{t}) \bigl ( \bar{U} \bigr )\cr
&\phantom{=}
+p(\bar{t})g_1(4\bar{t})\bar{x}\bar{U}
-2q(\bar{t})g_2(4\bar{t})\bar{x}\bar{U}\cr
&=\bigl(q(\bar{t})g_1(4\bar{t})-2\bar{t}p(\bar{t})g_2(4\bar{t})\bigr)\bar{U}
+\bigl(p(\bar{t})g_1(4\bar{t})-2q(\bar{t})g_2(4\bar{t})\bigr)\bar{x}\bar{U}\cr
\end{align*}
where $q(t):=\cosh(\sqrt{t})$ as in Equation (\ref{defqt}).

Thus we conclude 
\begin{equation} \label{mdef}
{\bf M} = \begin{pmatrix}&  g_1(4 \bar{t}) & - 2 g_2 (4 \bar{t})  \cr
&  q(\bar{t}) g_1(4 \bar{t})-2 \bar{t} p(\bar{t}) g_2 (4\bar{t})
& p(\bar{t}) g_1(4\bar{t}) - 2 q(\bar{t}) g_2(4 \bar{t}) \cr
\end{pmatrix}.
\end{equation}

\subsection{Computation of matrix $\bar{Q}$ and conclusion of proof of Theorem~\ref{KTF2r}}\label{subsection:finish proof}

In this section we prove that $\bar{Q}$ is the identity matrix,
and show the determinants of $\chT(Q)$ (and hence $Q$) are~$
1$. This allows us to conclude the proof of
Proposition~\ref{proposition:main technical} and Theorem~\ref{KTF2r}. 

We begin with the following. 

\begin{lemma}\label{lemma:barQ is identity}
The matrix $\bar{Q}$ is diagonal with integer entries. 
\end{lemma} 

\begin{proof} 
The elements $\Phi_{2r}^*\kappa^*(\bar{U})$ and
$\Phi_{2r}^*\kappa^*(\bar{x}\bar{U})$ are 
Weyl-invariant since $\tau$ is a $G$-bundle.
We know that
$$\Ker i^*\cap\{\mbox{symmetric polynomials}\}\cong\Z\qquad
\mbox{ in degree~$2(2r-1)$,}$$
generated by $\bar{K}_1$ and
$$\Ker i^*\cap\{\mbox{symmetric polynomials}\}\cong\Z\oplus\Z\qquad
\mbox{ in degree~$4r$,}$$
generated by $\bar{K}_2$ and~$\bar{b}\bar{K}_1$.
The intersections of the left hand sides of the above two equations with the
Weyl invariants are generated by $\bar{K}_1$ and~$\bar{K}_2$
respectively. 
Recall that $\bar{K}_1$ and $ \bar{K}_2$ are homogeneous.
Thus ${\Phi_{2r}}^*\kappa^*(\bar{U})=\lambda_1\bar{K}_1$
and ${\Phi_{2r}}^*\kappa^*(\bar{x}\bar{U})=\lambda_2\bar{K}_2$
for some integers~$\lambda_1$ and~$\lambda_2$.
\end{proof}

Let $\mathsf{U}$ be the Thom class of $\tau^{2r-1}$ in ordinary
cohomology. We next show that the diagonal matrix entries $\lambda_1$
and $\lambda_2$ (from the proof of the lemma above) are both~$1$,
by doing a computation in ordinary cohomology. 

\begin{lemma}\label{HimPhi}
The constants $\lambda_1$ and $\lambda_2$ appearing in the proof of Lemma~\ref{lemma:barQ is identity} are both equal to $1$. 
In particular, $\bar{Q}$ is the $2\times 2$ identity matrix. 
\end{lemma}

\begin{proof}
Let $\mathsf{K}_1$, $\mathsf{K}_2$, $\mathsf{x}$, $\mathsf{L}_j$ and
$\mathsf{s}_j$ be the ordinary cohomology
verions of $\bar{K}_1$, $\bar{K}_2$, $\bar{x}$, $\bar{L}_j$ and $\bar{s}_j$.
Using that $\mathsf{L}_1^2=0$ in the exterior
algebra $H^*\bigl((\PP^1)^{2r}\bigr)$ we get
$$\mathsf{x}\mathsf{K}_1=-\mathsf{L}_1\mathsf{s}_{2r-1}=-\mathsf{s}_{2r}
=\mathsf{K}_2,$$
and it follows that $\lambda_1=\lambda_2$.
 From Section 2 (\ref{exseq}), we know that
$\kappa^*(\mathsf{U})$ 
generates $H^*(F_{2r})$ in degree~$4r-2$.
(This is by the Thom isomorphism in ordinary cohomology.) 
Since ${\Phi_{2r}}^*\kappa^*(\mathsf{U})$
lies in
$\Ker i^*\cap\mbox{\{symmetric polynomials\}}$,
it is a
multiple of the ordinary cohomology version of the
generator~$\mathsf{K}_1$.
After dualizing, up to sign, the statement in the Lemma becomes equivalent
to the statement that $\kappa_*{\Phi_{2r}}_*$ is onto in 
homology in degree~$4r-2$.
Since $\kappa_*$ is an isomorphism in these degrees, we are asking whether
${\Phi_{2r}}_*$ is onto on~$H_{4r-2}(~)$.
Since $F_{2r}\rInto\Omega G$ is an isomorphism in these degrees, we may
instead consider the
composition $(\PP^1)^{2r}\rTo^{\Phi_{2r}} F_{2r}\rInto\Omega G$.
By construction, the composition
$(\PP^1)^{2r}\rTo^{\Phi_{2r}} F_{2r}\rInto\Omega G$
is the composition
$$(\PP^1)^{2r}\rTo^{i^{2r}} (\Omega G)^{2r}\rTo^{\rm mult}\Omega G$$
where $i:S^2\cong\PP^1\to\Omega G\cong\Omega S^3$ is the standard inclusion. 
Therefore $\Phi_{2r}^*$ can be computed from knowledge of $i_*$
and ${\rm mult}_*$.
The latter is the Pontrjagin multiplication in the Hopf
algebra $H_*(\Omega S^3)$, which is the polynomial algebra on the image, $y$, 
under~$i_*$ of a generator of~$H_2(S^2)$.
Since $y^{2r-1}$
is a group generator of this polynomial algebra in degree~$4r-2$
the map ${\Phi_{2r}}_*$ is surjective in degree~$4r-2$.
That is, $\lambda_1=\lambda_2=\pm1$.

The sign is not critical to us --- it does not even affect the
determinant --- so we omit the details, but finer analysis shows that $\lambda_1=\lambda_2=1$, as claimed. 
\end{proof}

We are now in a position to prove Proposition~\ref{proposition:main technical}. 

\begin{proof}[Proof of Proposition~\ref{proposition:main technical}]

We begin by computing some determinants. From
~\eqref{mdef} we have 
\begin{align*}
\det {\bf M}&=p(\bar{t})\bigl(g_1(4\bar{t})\bigr)^2
-2q(\bar{t})g_1(4\bar{t})g_2(4\bar{t})
+2q(\bar{t})g_1(4\bar{t})g_2(4\bar{t})
-4\bar{t}p(\bar{t})\bigl(g_2(4\bar{t})\bigr)^2\cr
&=p(\bar{t})\Bigl(\bigl(g_1(4\bar{t})\bigr)^2
-4\bar{t}\bigl(g_2(4\bar{t})\bigr)^2\Bigl).\cr
\end{align*}
From the definitions~\eqref{eq:def g1} and~\eqref{eq:def g2} of $g_1$ and $g_2$, it follows that 
$\bigl(g_1(4\bar{t})\bigr)^2-4\bar{t}\bigl(g_2(4\bar{t})\bigr)^2
% =\cosh^2\bigl((2r-1)\sqrt{z}/2\bigr)\left(\frac{\sinh
% (\sqrt{z}/2)}
% {\sqrt{z}/2}\right)^{2(2r-1)}\cr
% &\phantom{=}-z\frac{\sinh^2\bigl((2r-1)\sqrt{z}/2\bigr)}{z}
% \left(\frac{\sinh\sqrt{z}/2}
% {\sqrt{z}/2}\right)^{2(2r-1)}\cr
% &=\left(\frac{\sinh\sqrt{z}/2}
% {\sqrt{z}/2}\right)^{2(2r-1)}\cr
=\bigl(p(4\bar{t}/4)\bigr)^{2(2r-1)} = (p(\bar{t})^{2(2r-1)}$ 
so we conclude 
\begin{equation}\label{eq:det M}
\det {\bf M}=p(\bar{t})\bigl(p(\bar{t})\bigr)^{2(2r-1)}
=\bigl(p(\bar{t})\bigr)^{4r-1}.
\end{equation}

For the matrix ${\bf N}$, the description~\eqref{ndef} straightforwardly yields 
\begin{equation}\label{eq:det N}
\det {\bf N} = \bigl(p(\bar{t})\bigr)^{2r-1}\bigl(p(\bar{t})\bigr)^{2r}
=\bigl(p(t)\bigr)^{4r-1}.
\end{equation}

We know from~\eqref{matrixeqn} that 
\[
\det {\bf N} \det \chT(Q) = \det \bar{Q} \det {\bf M}
\]
in the (torsion-free) module $H^{\pi}_T(\pt;\Q)$. The above
computations of $\det {\bf M}$, $\det {\bf N}$, and $\bar{Q}$, together with
Lemma~\ref{HimPhi}, immediately yields
\[
p(\bar{t})^{4r-1} \det \bigl( \chT(Q)\bigr) = p(\bar{t})^{4r-1}
\]
from which it follows that $\det \bigl( \chT(Q)\bigr )  = 1$. Since $\chT$ is an
injective homomorphism we conclude $\det Q = 1$, hence in
particular $Q$ is invertible. This in turn implies that the images of
$U$ and $xU$ under $\Phi_{2r}^* \circ \kappa^*$ form an $R(T)$-module basis
for $\Ker i^* \cap \Bigl(K_T\bigl( (\PP^1)^{2r} \bigr)
\Bigr)^{S_{2r}}$. The statements of the proposition now follow. 
\end{proof} 

Next we turn to the proof of Theorem~\ref{KTF2r}. We will use the
following statement which can be proven using the monomial basis for
symmetric polynomials (see e.g. the classic text by MacDonald
\cite{Mac98}).

\begin{lemma}\label{lemma:macdonald}
Let $R$ be a ring and let $p(x_1,\ldots,x_n)\in R[x_1,\ldots,x_n]$ be a
symmetric polynomial which is linear in the variable $x_j$ for all~$j$.
Then $p(x_1,\ldots,x_n)$ is an $R$-linear combination of the
elementary symmetric polynomials 
$$s_0(x_1,\ldots,x_n),\ldots, s_n(x_1,\ldots, x_n). $$
\qed
\end{lemma}

\begin{proof}[Proof of Theorem~\ref{KTF2r}]

Consider the base case $r=0$. In this case the claim of the theorem is
immediate since $F_0 = \pt$. Now suppose by induction that the claim
of the theorem holds for $r-1$. We wish to prove the claim holds for
$r$.

Due to the quadratic relation $L_j^2 = vL_j-1$ 
in Theorem~\ref{theorem:P1summary}, any symmetric polynomial in the
$L_j$ 
in $K^*_T\bigl((\PP^1)^{2r}\bigr)$ can be
expressed as a symmetric polynomial which is linear in each $L_j$. By
Lemma~\ref{lemma:macdonald} it also follows that it can be written as
an $R(T)$-linear combination of $s_0, \ldots, s_{2r}$.
It follows that
\begin{equation*}
\begin{split}
\bigl(K^*_T((\PP^1)^{2r})\bigr)^{S_{2r}} & = \mbox{the $K^*_T(\pt)$-subalgebra
of $K^*_T\bigl((\PP^1)^{2r}\bigr)$ generated
by $ s_1,\ldots,s_{2r}$} \\ 
&= \mbox{the $K^*_T(\pt)$-submodule
of $K^*_T\bigl((\PP^1)^{2r}\bigr)$ generated
by $s_0,\ldots,s_{2r}$}. \\ 
\end{split} 
\end{equation*}

Lemma~\ref{lemma:inclusion} shows $$\Phi_{2r}^*\bigl(K^*_T(F_{2r})\bigr)
\subseteq (K^*_T((\PP^1)^{2r}))^{S_{2r}},$$ and by Lemma~\ref{lemma:sufficient}, the
claim proven in Proposition~\ref{proposition:main technical} suffices
to show the containment in the other direction. Since $\Phi_{2r}^*$ is
injective by Proposition~\ref{proposition:injective in KG}, the claim of the
theorem now follows. 

\end{proof}

We also record the $G$-equivariant version. 

\begin{theorem}\label{KF2r}
Let $G=SU(2)$. Let $\Phi_{2r}:
(\PP^1)^{2r} \to F_{2r}$ be the map given in
Definition~\ref{definition:Phi 2r}. Then 
\begin{align*}
K^*_G(F_{2r})\cong\Phi_{2r}^*\bigl(K^*_G(F_{2r})\bigr)
&=\{\mbox{symmetric polynomials in~$K^*_G\bigl((\PP^1)^{2r}\bigr)$}\}\cr
&= \mbox{the $K^*_G(\pt)$-subalgebra
of $K^*_G\bigl((\PP^1)^{2r}\bigr)$ generated
by $s_1,\ldots,s_{2r}$}\cr
&= \mbox{the $K^*_G(\pt)$-submodule
of $K^*_G\bigl((\PP^1)^{2r}\bigr)$ generated
by $s_0,\ldots,s_{2r}$}.\cr
\end{align*}
\end{theorem}

\begin{proof}
Since $\Phi_{2r}^*$ is an algebra injection, the theorem is 
equivalent to the statement that
$\Phi^*_{2r}\bigl(K^*_G(F_{2r})\bigr)$ is the subalgebra
of $K^*_G\bigl((\PP^1)^n\bigr)$ generated by $s_0, s_1,\ldots,s_{2r}$.
We know that 
$K^*_G(F_{2r})\cong\bigl(K^*_T(F_{2r})\bigr)^W$ from \cite{HJS12}. 
In Section~\ref{subsection:BottP} we showed that
$$K^*_G\bigl((\PP^1)^n\bigr)\cong \Bigl(K^*_T\bigl((\PP^1)^n\bigr)\Bigr)^W.$$
Therefore,  by taking Weyl-invariants, we conclude that
$\Phi^*_{2r}\bigl(K^*_G(F_{2r})\bigr)$ is the submodule
of $K^*_G\bigl((\PP^1)^{2r}\bigr)$ generated by $s_0,\ldots,s_{2r}$. The
other statements follow.
\end{proof}

\section{Proof of the main theorem}\label{section:proof of main}

The following is now an immediate consequence of the previous results.

\begin{theorem}\label{maintheorem}
Let $G=SU(2)$ and let $\Omega G$ be the space of (continuous) based
loops in $G$, equipped with the natural $G$-action by pointwise
conjugation. Then 
$$K^*_G(\Omega G)
=\varprojlim_r\, 
\Bigl(K^*_G\bigl((\PP^1)^{2r}\bigr)\Bigr)^{S_{2r}}
=\varprojlim_r\, 
\{\mbox{symmetric polynomials in~$K^*_G\bigl((\PP^1)^{2r}\bigr)$}\}
$$
where 
$K^*_G\bigl((\PP^1)^{2r}\bigr)\cong 
R(G)[L_1, \dots, L_{2r}]/I,$
and $I$ is the ideal generated by $\{L_j^2-vL_j+1\}_{j=1}^n$.
Here $R(G)$ is the representation ring of $G$,
$v$ is the standard representation of $G = SU(2)$ on~${\CC}^2 $ and $L_j$ is
the pullback of either the canonical line bundle over the
$j$th factor of~$(\PP^1)^{2r}$ or its inverse, depending on~$j$ (see
Definition \ref{definition:L_j and barL_j}). 
The maps in this inverse system are given by
$$i^* (s_j) = \begin{cases}
s'_0 &\mbox{if $j=0$};\cr
s'_1 +v s'_{0} &\mbox{if $j=1$};\cr
 s'_j +v s'_{j-1} + s'_{j-2}
& \mbox{if $1 < j \le 2r-2$};\cr
  v s'_{2r-2} + s'_{2r-3}
& \mbox{if $ j = 2r-1$};\cr
       s'_{2r-2}
& \mbox{if $    j =2r$},\cr
\end{cases}
$$
where $s_j$ and $s'_j$ are respectively
the $j$th elementary symmetric polynomials in
$\{L_1,\ldots, L_{2r}\}$ and 
$\{L'_1,\ldots, L'_{2r-1}\}$.
\end{theorem}

\def\cprime{$'$}

\end{document}